\documentclass[a4paper,11pt]{article}
\usepackage{fancyhdr}
\usepackage{amsbsy,amsfonts,amssymb,amsmath}
\usepackage{amsthm}
\usepackage{stmaryrd} % used only to define \enum et \enu
\usepackage[active]{srcltx}
\usepackage{graphicx}
\graphicspath{{../Figures/}}
\usepackage{color}
\definecolor{bleuf}{rgb}{0.1,0.1,0.6}
\usepackage{multicol}
 \usepackage{hyperref}
 \hypersetup{
        colorlinks = true,           
        breaklinks = true,          
        urlcolor = bleuf,           
        linkcolor =bleuf,                  
        citecolor = bleuf,                     
        linktocpage = true,
        pdftitle={Markovian loop clusters on the complete graph and coagulation equations},   
        pdfauthor={S. Lemaire},                  
     }
\makeatletter
\renewcommand{\section}{\@startsection{section}{1}{0mm}%
{\baselineskip}{\baselineskip}%
{\large\bfseries\centering}}

\renewcommand{\subsection}{\@startsection{subsection}{2}{0mm}%
{0.5\baselineskip}{0.2\baselineskip}%
{\large\bfseries}}

\renewcommand{\subsubsection}{\@startsection{subsubsection}{1}{0mm}%
{0.4\baselineskip}{0.1\baselineskip}%
{\slshape\bfseries}}
\renewcommand{\subsubsection}{\@startsection{subsubsection}{1}{0mm}%
{0.4\baselineskip}{0.1\baselineskip}%
{\slshape\bfseries}}
\makeatletter

\DeclareMathOperator{\NN}{\mathbb{N}}

\DeclareMathOperator{\RR}{\mathbb{R}}
\DeclareMathOperator{\Pd}{\mathbb{P}}
\DeclareMathOperator{\Ed}{\mathbb{E}}
\DeclareMathOperator{\un}{1\hspace{-.29em}I}

\newcommand{\intervalle }[4]{\mathopen{#1}#2
\mathclose{}\mathpunct{};#3
\mathclose{#4}}

\newcommand\enu[1]{{\intervalle\llbracket {1}{#1}\rrbracket}}
\newcommand\KN[1][n]{\llbracket{#1}\rrbracket}

\newcommand\DL{\mathcal{DL}}

\newcommand\NL{\mathcal{DL}}
\newcommand\Vois{\mathcal{N}}
\newcommand\ER{\mathcal{H}}

\newcommand\mup{\dot{\mu}}

\newcommand\Pois{\text{Poisson}}
\newcommand\CPois{\text{CPois}}
\newcommand\Geo{\mathcal{G}_{\NN^{*}}(\frac{\epsilon}{\epsilon+1})}
\DeclareMathOperator{\et}{ and }

\DeclareMathOperator{\TV}{d_{\text{TV}}}

\DeclareMathOperator{\var}{Var}
\newcommand\inter{\;\text{intersects}\;}

\renewcommand{\epsilon}{\varepsilon}
\newcommand\iid{i.i.d.\ }
\newtheorem{thm}{Theorem}[section]
\newtheorem{lem}{Lemma}[section]
\newtheorem{prop}{Proposition}[section]
\newtheorem{corol}{Corollary}[section]

\newtheorem*{bibli}{Theorem}
\theoremstyle{definition}
\newtheorem{ex}{Example}

\newtheorem{rem}{Remark}[section]
\title{Markovian loop clusters on the complete graph\\ and coagulation equations}
\date{June 2014}
\author{S. Lemaire\thanks{Universit\'e Paris-Sud, Laboratoire de Math\'ematiques, UMR 8628, Orsay
{F-91405}; CNRS, Orsay, {F-91405}; email: \url{sophie.lemaire@math.u-psud.fr}}}

\begin{document}

\maketitle
\begin{abstract}
Poissonian ensembles of Markov loops on a finite graph define a random
graph process in which the addition of a loop can merge  more than  two
connected components. We study Markov loops on the complete graph derived from  
a simple random walk  killed at each step with a constant probability.
Using a component exploration procedure,  we describe  the
asymptotic distribution of the  connected component size of a vertex at a
time proportional to the number of vertices, show that the largest component size 
undergoes a phase transition and establish the coagulation equations associated to 
this random graph process. 
\end{abstract}\medskip
\textit{AMS 2000 Subject Classification.} Primary 60C05. Secondary 05C80, 60J80, 82C31\medskip\\
\textit{Key words.} Poisson point process of loops, random graph process, branching processes,  
coalescent process, coagulation equations. 
\newpage
\section*{Introduction}
We consider the Poissonian ensembles of Markov loops on the complete graph $K_n$
  derived from  a simple random walk  killed at a constant rate
$\kappa_n=n\epsilon$ (loops can be viewed as excursions of the simple random walk
with a random starting point, up to re-rooting). Poissonian ensembles are seen as a Poisson point process
of loops indexed by `time'. The edges crossed by loops before  time $t$ define
a subgraph $\mathcal{G}^{(n)}_{t}$ of $K_n$. The evolution of the connected
components of $\mathcal{G}^{(n)}_{t}$ also named \emph{loop clusters}, defines a
 coalescent process with multiple mergers: the addition of a loop of length $k$ can merge up to $k$ clusters  at the same time. 

The notion of Poissonian ensembles of Markov loops (loop soups) was introduced
by Lawler and Werner in \cite{lw}  in the context of two
dimensional  Brownian motion (it already appeared informally 
in \cite{sym}). Loop clusters induced by a Brownian loop soup were used to give a
construction of  conformal loop ensembles (CLE) in \cite{WernerCRAS03} and
\cite{shw}. Some general
properties of  loop clusters on  finite and countable graphs were
presented in \cite{LeJanLemaire13} 
(see  \cite{stfl},  \cite{lawlerlimic}  and \cite{Sznitman} for  studies of 
other aspects of  Markov loops on graphs). 

The random graph process $(\mathcal{G}^{(n)}_{t})_{t\geq 0}$ depends on the `killing parameter'  $\epsilon$.  
For large values of $n$ and $\epsilon$,  the proportion of 
loops of length 2 is near 1, hence we can expect that $(\mathcal{G}^{(n)}_{t})_{t\geq 0}$ behaves like  the  random graph process introduced 
by Erd\"os and R\'enyi in \cite{ErdosRenyi59}.  In  \cite{LeJanLemaire13}, the asymptotic distribution  of the cover time 
(the first time when  $\mathcal{G}^{(n)}_{t}$ has no isolated vertex) and the  coalescence time 
(the first time when  $\mathcal{G}^{(n)}_{t}$ is connected) 
were studied   showing in particular that $\epsilon(\epsilon+1)n\log(n)$ is a
sharp threshold function for connectivity of the random graph process
$(\mathcal{G}^{(n)}_{t})_{t\geq 0}$. In comparison, the threshold function 
for the connectivity of the Erd\"os-R\'enyi random graph with $n$ vertices is $\frac{1}{2}n\log(n)$ (see \cite{ErdosRenyi59}\footnote{More precisely, if  $G(n,N)$ denotes a random graph obtained by forming $N$ links between $n$ labelled vertices, each of the  $\displaystyle{\binom{N}{\binom{n}{2}}}$ graphs being equally likely, then the probability that   ${G(n,\lfloor\frac{n}{2}(\log(n)+c)\rfloor)}$ is a connected graph  converges to 
$\displaystyle{\exp(-e^{-c})}$ as $n$ tends to $+\infty$ for every $c\in\RR$.}).
In this paper, we  study the  (connected) component sizes of $\mathcal{G}^{(n)}_t$ and the hydrodynamic behavior 
of the associated coalescent process. 

Let us review some properties related to our study in the case of
 the  Erd\"os-R\'enyi random graph. We follow the  presentation given by J. Bertoin in \cite{Bertoinbookfrag}, chapter 5. To make the link with the multiplicative coalescent simpler, we consider a variant of the Erd\"os-R\'enyi random graph  constructed as follows: 
let $\{X_{\{x,y\}},\ x\neq y\}$ be a family of independent random variables  indexed by the 
edges of $K_n$ with Exponential(1)-distribution. We define an increasing family of random graphs
denoted $(\ER(n,t))_{t\geq 0}$ by setting: 
$e$ is an edge of $\ER(n,t)$ if and only if $X_e\leq t$. \\
Let $c^{(n)}_t(x)$ denote the component size of a vertex $x$  of  $\ER(n,\frac{t}{n})$ and let 
$c^{(n)}_{1,t}\geq c^{(n)}_{2,t}$ denote the two largest component sizes. 
\paragraph{Phase transition} 
\begin{enumerate}
\item Assume that $t<1$. 
\begin{itemize}
\item For every vertex $x$, $c^{(n)}_t(x)$ converges in distribution to the total population 
size of a Galton-Watson process with one
progenitor and \Pois($t$) offspring distribution. 
\item  Let $I_{t}$ be the value at 1 of the Cram\'er
function of the \Pois$(t)$-distribution:
$I_{t}=t-1-\log(t)$.  \\
\emph{For every $a>I^{-1}_{t}$,  $\Pd(c^{(n)}_{1,t} \geq
a\log(n))$ converges to $0$ as $n$ tends to $+\infty$}. 
\end{itemize}
\item Assume that $t>1$ and denote by $q_t$ the extinction probability of a Galton-Watson process with one
progenitor and \Pois($t$) offspring distribution.\\
\emph{For every $a\in]1/2,1[$, there exist $b>0$ and $c>0$ such that 
$$\Pd[|c^{(n)}_{1,t}-(1-q_t)n|\geq n^a]+\Pd[c^{(n)}_{2,t}\geq c\log(n)]=O(n^{-b}).$$} 
\end{enumerate}
This phase transition was first proved by Erd\"os and R\'enyi in \cite{ErdosRenyi60}. The statements
we present are taken from \cite{bookvanderHofstad} where proofs are based on the use of 
branching processes. 
\paragraph{Coalescent process.} 
The component sizes of $\ER(n,\frac{t}{n})$ evolve as a  
multiplicative coalescent process with binary aggregations: aggregations  
of more than two components at one time do not occur and the aggregation rate 
of two components is proportional to the product of their sizes. 
\begin{itemize}
\item For any  $t>0$ and $x\in \NN^*$, the average number of components of  size $x$ 
in $\ER(n,\frac{t}{n})$ converges in $L^2$ to 
$$n(t,x)= \frac{(tx)^{x-1}e^{-tx}}{x.x!}\ \forall x\in\NN^* \text{ and }t\in[0,1].$$ 
The value $xn(x,t)$ is equal to the probability that $x$ is the total population size of  
a Galton-Watson process with
one progenitor and \Pois($t$) offspring distribution\footnote{For $t\leq 1$,  $\{xn(x,t), x\in\NN^*\}$ 
is a probability distribution called Borel-Tanner distribution with parameter $t$}.
\item $\{n(x,\cdot),\ x\in \NN^*\}$ is  solution on $\RR_+$ of the Flory's coagulation equations with 
multiplicative kernel:
\begin{multline}
\label{Floryeq}
\frac{d}{dt}n(x,t)=\frac{1}{2}\sum_{y=1}^{x-1}y(x-y)n(y,t)n(x-y,t)\\
-\sum_{y=1}^{+\infty}xyn(t,x)n(t,y)- xn(t,x)\sum_{y=1}^{+\infty}y(n(0,y)-n(t,y))
\end{multline}
Up to time 1, this solution coincides with the solution of the 
Smoluchowski's coagulation equations with multiplicative kernel starting from  the monodisperse state:  
\begin{equation}
\label{Smoluchowskieq}
\frac{d}{dt}n(x,t)=\frac{1}{2}\sum_{y=1}^{x-1}y(x-y)n(y,t)n(x-y,t)-xn(t,x)\sum_{y=1}^{+\infty}yn(t,y).
\end{equation}
Equations \eqref{Smoluchowskieq} introduced by Smoluchowski in \cite{Smolu16} are used for example 
to describe aggregations of polymers in an homogeneous medium where diffusion effects are ignored.  
The first term in the right-hand side  describes the formation of a  particle of mass $x$ by
aggregation of two particles, the second sum describes the ways  a particle of
mass $x$ can  be aggregated with  another particle. If the total mass of particles decreases after 
a finite time, the system is said to  exhibit a `phase transition' called `gelation': the loss of 
mass is interpreted as the formation of infinite mass particles called  gel. Smoluchowski's equations 
do not take into account interactions between gel and finite mass particles.  
Equations \eqref{Floryeq} introduced by Flory in \cite{Flory} are a modified version of the 
Smoluchowski's equations with   an extra term describing the loss of a particle of mass $x$ by 
`absorption' in the gel. 
Let $T_{gel}$ denote the largest time such that the Smoluchowski's coagulation
equations with monodisperse initial condition has a solution which has the
mass-conserving property\footnote{Different definitions of the `gelation time' $T_{gel}$ are used 
in the literature:  the gelation time is sometimes defined as the smallest time when the second 
moment diverges (see \cite{Aldousreview})}. Then $T_{gel}=1$ and $T_{gel}$ coincides with the 
smallest time when  the second moment $\sum_{x=1}^{+\infty}x^2n(x,t)$ diverges (see \cite{McLeod}).
Let us note that the random graph process $(\ER(n;\frac{t}{n}))_{t\geq 0}$ is equivalent to the 
microscopic model  introduced by Marcus \cite{Marcus} and further studied by Lushnikov \cite{Lushnikov} 
(see \cite{BuffetPule} for a first study of the relationship between  these two models and 
\cite{Aldousreview} for a review, \cite{Norris99}, \cite{Norris00} and \cite{FournierGiet} 
for convergence results of  Marcus-Lushnikov's model to \eqref{Floryeq}). 
\end{itemize}
 
The aim of this paper is to show that similar statements hold for the loop 
model  if the \Pois$(t)$-distribution is replaced  by the compound Poisson distribution 
with probability-genera\-ting function $s\mapsto \exp(\frac{t(1-s)}{\epsilon(\epsilon+1-s)})$ 
(i.e. the distribution of a \Pois($\frac{t}{\epsilon(\epsilon+1)}$)-distributed number of 
independent  random variables with geometric distribution on $\NN^*$ of  parameter 
$\frac{\epsilon}{\epsilon+1}$). Phase transition in the loop model occurs at $t=\epsilon^2$. 
Hydrodynamic behavior of the coalescent process 
associated to the loop model is described by new coagulation equations in which 
more than two particles can collide at the same time. 

Section \ref{sect:model} is devoted to a presentation of the Markov loop model on the 
complete graph and the statement of the main results (Theorem \ref{thmsub}, Theorem \ref{thmtransition} and 
Proposition \ref{prop:hydrodyn}). 
In section \ref{sect:exploration}, we describe the component exploration process used to compute 
the component size of a vertex and to  construct the associated Galton-Watson process.  
The asymptotic distribution of the component size of a vertex  is studied in Section \ref{sect:proofth} and 
the proof of Theorem \ref{thmsub} is presented.  
In Section \ref{sect:transition}, we prove Theorem \ref{thmtransition} which presents some properties of 
the largest component size in the two phases, $t <\epsilon^2$
and $t>\epsilon$. 
Section \ref{sect:coageq} is devoted to the proof of Proposition \ref{prop:hydrodyn} that 
describes the hydrodynamic behaviour of the coalescent process. 
\section{Description of the model and main results\label{sect:model}}
\subsection{Setting}
 We consider the complete graph $K_n$ with $n$ vertices. The set of vertices is
identified with $\KN=\{1,\ldots,n\}$. We add to each vertex $x$ a self-loop $\{x,x\}$. This
defines an undirected graph denoted $\bar{K}_n$. We consider the loop sets induced by a simple random walk on
$\bar{K}_n$ killed at each step with probability $\frac{\epsilon}{\epsilon+1}$
with $\epsilon>0$. In other words, the graph $\bar{K}_n$ is endowed with  unit
conductances and a uniform killing measure with intensity $\kappa_n=n\epsilon $.
The transition matrix $P$ of the random walk is defined by
$P_{x,y}=\frac{1}{n(\epsilon+1)}$ for every $x,y\in\KN$. 

A discrete based loop
$\ell$ of length $k\in\NN^*$ on $\bar{K}_n$ is defined as an element of
$\KN^k$.
A discrete loop is an equivalent class of based loops for the following
equivalent relation:  the based loop of length $k$, $(x_1, \ldots, x_k)$ is
equivalent to the based loop of length $k$
$(x_i,\ldots,x_{k},x_1,\ldots,x_{i-1})$ for every $i\in\{2,\ldots,
k\}$.  We associate to each based loop $\ell=(x_1,\ldots,x_k)$ of length $k\geq
2$ the weight $\mup(\ell)=\frac{1}{k}P_{x_1,x_2}P_{x_2,x_3}\ldots
P_{x_k,x_1}=\frac{1}{k(n(\epsilon+1))^k}$. This defines a measure $\mup$ on the
set of discrete based loops of length at least 2,  which is invariant by the
shift and therefore  induces a measure $\mu$ on the set of discrete loops of
length at least 2 denoted by $\DL(\KN)$. 
The Poisson loop sets on $\bar{K}_n$ is defined as a Poisson point process
$\mathcal{DP}$ with intensity $\text{Leb}\times\mu$  on $\RR_+\otimes \DL(\KN)$.
For $t>0$, we denote by  $\DL^{(n)}_{t}$ the projection of the set 
$\mathcal{DP}\cap ([0,t]\times \DL(\KN))$. The loop set $\DL^{(n)}_{t}$ on $\bar{K}_n$
defines a subgraph denoted by $\mathcal{G}^{(n)}_t$. 

Let $C^{(n)}_{t}(x)$ denote
the connected component of the vertex $x$ in the random graph
$\mathcal{G}^{(n)}_t$.  The aim of the paper is to study the size of
$C^{(n)}_{nt}(x)$ as $n$ tends to $+\infty$. 
The loop sets defined by this model are slighly different from the loop sets on
the complete graph studied in \cite{LeJanLemaire13} since we add a self-loop
$\{x,x\}$ at each vertex $x$. But the partitions induced by the loop clusters on
$\bar{K_n}$ or $K_n$ have the same distribution recalled in the following
proposition: 
\begin{prop}
\label{semigroupC}
For every $t>0$, let $\mathcal{C}_t$ denote the partition induced by the
connected components of   $\mathcal{G}^{(n)}_t$ and let $\tilde{\mathcal{C}}_t$
denote the partition induced by the loop sets on $K_n$ associated to the
transition matrix $\tilde{P}$ defined by:
$\tilde{P}_{x,y}=\frac{1}{n-1+n\epsilon}\un_{\{x\neq y\}}$ for every $x,
y\in\KN$.   The Markov processes $(\mathcal{C}_t)_t$ and
$(\tilde{\mathcal{C}}_t)_t$ have the same distribution: 
\begin{itemize}
 \item 
 If  $\pi$ is a partition of the set of vertices $\KN$ with $k$ blocks
$B_1,\ldots,B_k$ then the probability that $\mathcal{C}_{t}$ is finer than $
\pi$ is $$\Pd_{\pi_{0}}(\mathcal{C}_{t}\preceq
\pi)=(\frac{\epsilon}{\epsilon+1})^{t}\prod_{i\in
I}(1-\frac{|B_i|}{n(\epsilon+1)})^{-t}\un_{\{\pi_0 \preceq\pi\}}.$$
\item 
 From state $\pi=\{B_i,\; i\in I\}\in \mathfrak{P}(\enu{n})$, the only possible
transitions of $(\mathcal{C}_{t})_ {t\geq 0}$ are to partitions
$\pi^{\oplus J}$ obtained by merging blocks indexed by a subset  $J$ of  $I$ with $L\geq
2$ elements. The transition rate from $\pi$ to  $\pi^{\oplus J}$ is equal to: 
\begin{align}
\tau^{(n)}_{\pi,\pi^{\oplus J}}&=\sum_{k\geq
L}\frac{1}{kn^k(\epsilon+1)^k}\sum_{(i_1,\ldots,i_k)\in W_k(J)}
\prod_{u=1}^{k}|B_{i_u}|\label{eqtransitionKn1}\\
&=\sum_{k\geq
L}\frac{1}{kn^k(\epsilon+1)^k}\sum_{\substack{(k_1,\ldots,k_L)\in(\NN^*)^{L},
\\k_1+\ldots
+k_{L}=k}}\binom{k}{k_1,\ldots,k_{L}}\prod_{u=1}^{L}|B_{j_u}|^{k_{u}}\label{
eqtransitionKn2}
\end{align}
where $W_k(J)$ is the set of $k$-tuples of  $J$ in which each element of $J$
appears. 
\end{itemize}
\end{prop}
Proposition \ref{semigroupC}  can be proved as in  \cite{LeJanLemaire13} where
the distribution of $(\tilde{\mathcal{C}}_t)_{t> 0}$ is computed. 
\subsection{Main results}
In the Erd\"os-R\'enyi random graph process, the connected component  of a
vertex $x$  can be compared  with a Galton-Watson process with ancestor $x$ in
which any vertices $y$ connected by an edge to a vertex $z$ are seen as
offspring of $z$. 
 Unlike  edges in Erd\"os-R\'enyi random graph, loops can cross a vertex several
times and intersect other loops at several vertices. But 
theses events are rare enough so that the component $C^{(n)}_{nt}(x)$ for large $n$ can
still be compared  with a Galton-Watson process with ancestor $x$ in which 
offspring of a vertex $z$ corresponds to vertices $y\neq z$ crossed by  loops in 
$\DL^{(n)}_{nt}$ that pass through $z$ and  the number of offspring can be
approximated by $\displaystyle{\sum_{\ell \in \mathcal{DL}^{(n)}_{nt}\; \text{s. t.}\; x\in \ell}
(|\ell|-1)}$. 

\noindent Before stating the main results, let us introduce some notations.
\begin{itemize}
 \item 
 For a positive real $\lambda$ and a probability distribution $\nu$ on $\RR$, let
$\CPois(\lambda,\nu)$ denote the compound Poisson distribution with parameters
$\lambda$ and $\nu$:  $\CPois(\lambda,\nu)$  is the probability distribution of $\sum_{i=1}^{N}X_i$,
where  $N$ is a Poisson distributed random variable with expected value
$\lambda$ and $(X_i)_i$ is a sequence of independent random variables with law
$\nu$ that are independent of $N$. 
\item For $p\in[0,1]$, let $\mathcal{G}_{\NN^*}(p)$ denote the geometric distribution 
on $\NN^*$ with parameter $p$ (its probability mass function is $(1-p)^{k-1}p\quad \forall k\in\NN^*$).
\item For $u\in\NN^*$, $\epsilon>0$ and $t>0$, let $T^{(u)}_{\epsilon,t}$ denote 
the total number of descendants  of a Galton-Watson
process with family size distribution 
$\CPois(\frac{t}{\epsilon(\epsilon+1)},\mathcal{G}_{\NN^*}(\frac{\epsilon}{
\epsilon+1}))$ and $u$ ancestors. 
\end{itemize}
\subsubsection{Component sizes}
The following theorem shows in particular, that for $t\in[0;\epsilon^2]$ the component 
size of a vertex at time $nt$  converges in distribution to $T^{(1)}_{\epsilon,t}$:
\begin{thm} 
\label{thmsub}
Let $\epsilon$ and  $t$ be two positive reals. Let $(k_n)_n$ be a  sequence of positive 
numbers.    
 $$\Pd(|C^{(n)}_{nt}(x)|\leq k_n)-\Pd(T^{(1)}_{\epsilon,t}\leq k_n)=O(\frac{k^{2}_n}{n}).$$ 
\end{thm}
\begin{rem}\
Let $t$ and $\epsilon$ be two positive reals. 
\begin{enumerate}
\item A Galton-Watson process with family size distribution 
$\CPois(\frac{t}{\epsilon(\epsilon+1)},\mathcal{G}_{\NN^*}(\frac{\epsilon}{
\epsilon+1}))$ is subcritical if and only if $t<\epsilon^2$. Let
$q_{\epsilon,t}$ denote the extinction probability of  such a Galton-Watson
process starting with one ancestor. It is a decreasing function of $t$ and an
increasing function of $\epsilon$. \\ 
Moreover,
\begin{equation}
\label{lawTGW}
\left\{\begin{array}{l}\Pd(T^{(u)}_{\epsilon,t} = u) = \displaystyle{e^{-\frac{u
t}{\epsilon(\epsilon+1)}}}\\
\Pd(T^{(u)}_{\epsilon,t} = k) = \displaystyle{\frac{u}{k}\frac{e^{-\frac{k
t}{\epsilon(\epsilon+1)}}}{(\epsilon+1)^{k-u}}\sum_{j=1}^{k-u}\binom{k-u-1}{j-1}
\frac{1}{j!}\Big(\frac{kt}{\epsilon+1}\Big)^j}\quad \forall k\geq u+1.  
\end{array}\right.
\end{equation}
For $t\leq \epsilon^2$, $T^{(u)}_{\epsilon,t}$ is almost surely finite and for 
$t>\epsilon^2$, $P(T^{(u)}_{\epsilon,t}<\infty)=q_{\epsilon,t}^{u}<1$. 
(see appendix \ref{propGW} for a detailed description of the properties of such a Galton-Watson process.)

\item 
The convergence result in Theorem \ref{thmsub} still holds if  $t$ and
$\epsilon$ are replaced in the statement by two positive sequences $(t_n)_n$ and
$(\epsilon_n)_n$ that converge to $t$ and $\epsilon$ respectively. 
\end{enumerate}
\end{rem}
Theorem \ref{thmsub} is used to show that the 
component sizes of $(\mathcal{G}^{(n)}_{nt})_{t\geq 0}$ undergo
a phase transition at $t=\epsilon^2$ similar to the phase transition of the 
Erd\"os-R\'enyi random graph process.  
\begin{thm}\label{thmtransition} Let $C^{(n)}_{nt,1}$ and $C^{(n)}_{nt,2}$ denote the first and second largest 
components of the random graph $\mathcal{G}^{(n)}_{nt}$. 
\begin{enumerate}
\item \textbf{Subcritical regime.} 
\label{subcrit}
 Assume that $0<t<\epsilon^2$. Set
$h(t)=\sup_{\theta\in]0,\log(\epsilon+1)[}(\theta-\log(L_{t,\epsilon}(\theta)))$
where $L_{t,\epsilon}$ is the  moment-generating function of the compound Poisson
distribution\linebreak  ${\CPois(\frac{t}{\epsilon(\epsilon+1)},\Geo)}$.\footnote{$h(t)$ is
the value of the Cram\'er function at 1 of
$\CPois(\frac{t}{\epsilon(\epsilon+1)},\Geo)$. As the expectation of
$\CPois(\frac{t}{\epsilon(\epsilon+1)},\Geo)$ is $\frac{t}{\epsilon^2}$, $h(t)$ 
is positive for $t< \epsilon^2$ and vanishes at $t=\epsilon^2$.} \\
For every $a>1/h(t)$, 
$\Pd(|C_{nt,1}^{(n)}|> a\log(n))$ converges to $0$ as $n$ tends
to $+\infty$. 
\item \textbf{Supercritical regime.}
Assume that $t>\epsilon^2$ and  let $q_{\epsilon,t}$ denote the extinction probability of a 
Galton-Watson process with one
progenitor and $\CPois(\frac{t}{\epsilon(\epsilon+1)},\Geo)$ offspring distribution.\\
For every $a\in]1/2,1[$, there exist $b>0$ and $c>0$ such that 
$$\Pd[||C^{(n)}_{nt,1}|-(1-q_{\epsilon, t})n|\geq n^a]+\Pd[|C^{(n)}_{nt,2}|\geq c\log(n)]=O(n^{-b}).$$ 
\end{enumerate}
\end{thm}

\subsubsection{Coagulation equations}
We turn  to the hydrodynamic behavior of the coalescent process generated
by the loop sets.  Components of size $k$ can be seen as a cluster of $k$ particles of unit mass; 
at the same time, several clusters of masses $k_1,\ldots, k_j$ can merge into a unique cluster 
of mass $k_1+\ldots+k_j$ at a rate proportional to the product $k_1\ldots k_j$.  The initial 
state corresponds to the monodisperse configuration ($n$ particles of unit mass).  
The following proposition describes the asymptotic limit of  the average number of
components of size $k$  at time $nt$ as the number of particles $n$ tends to $+\infty$:
\begin{prop}
\label{prop:hydrodyn}
For $k\in \NN^*$, $n\in \NN$ and $t>0$, let $\rho_{\epsilon,t}^{(n)}(k) =
\frac{1}{nk}| \{x\in \KN,\ |C^{(n)}_{nt}(x)| = k\}|$ be the average number of
components of size $k$  and let $\rho_{\epsilon,t}(k) =\frac{1}{k}P(T^{(1)}_{\epsilon, t}
= k)$.
\begin{enumerate}
\item $(\rho^{(n)}_{\epsilon,t}(k))_n$ converges to $\rho_{\epsilon,t}(k)$ in $L^2$ for every
$t>0$. 
\item $(\rho_{\epsilon, t}(k),\ k\in \NN^* \text{ and } t\geq 0)$ is a
solution to the following coagulation equations:
\begin{equation}
\label{coageq}
\frac{d}{dt}\rho_t(k) =
\sum_{j=2}^{+\infty}\frac{1}{(\epsilon+1)^j}G_{j}(\rho_t,k)
\end{equation}
where 
\begin{multline}
\label{Gj}
G_{j}(\rho_t,k)=
\frac{1}{j}\Big(\sum_{\substack{(i_1,\ldots,i_{j})\in(\NN^*)^{j}\\  i_1+\cdots
+i_j=k}}\prod_{u=1}^{j}i_u\rho_{t}(i_u)\Big)\un_{\{j\leq
k\}}-k\rho_t(k)\Big(\sum_{i=1}^{+\infty}
i\rho_t(i)\Big)^{j-1}\\
-k\rho_{t}(k)\sum_{h=1}^{j-1}\binom{j-1}{h}\Big(\sum_{i=1}^{+\infty}i(\rho_{0}(i)-\rho_t(i))\Big)^h \Big(\sum_{u=1}^{+\infty}
u\rho_t(u)\Big)^{j-1-h}
\end{multline}
\end{enumerate} 
\end{prop}
\begin{rem}\
\begin{enumerate}
\item Consider a medium with integer mass particles and let $\rho_{t}(k)$ denote
the density of mass $k$ particles at time $t$. Equation \eqref{coageq} describes
the evolution of $\rho_{t}(k)$ if for every $j\geq 2$ the number of aggregations
of $j$ particles of mass $i_1,\ldots,i_j$ in time interval $[t,t+dt]$ is assumed
to be $$\frac{1}{j(\epsilon+1)^j}\rho_t(i_1)\ldots\rho_t(i_j)K_j(i_1,\ldots,
i_j) dt$$ 
where $K_j(i_1,\ldots,i_j)=i_1\cdots i_j$ is the multiplicative kernel. \\
The first sum of $G_j$ describes the formation of a particle of mass $k$ by
aggregation of $j$ particles, the second sum describes the ways  a particle of
mass $k$ can  be aggregated with $j-1$ other particles. The third term in the definition of $G_j$ is null if the total mass is preserved. Otherwise, the decrease of the total mass can be interpreted as the appearance of a `gel' and the  third term  describes the different ways a particle of mass $k$ can be aggregated with the gel and other particles. 
\item $\rho_{\epsilon,t}$ defined by  $\rho_{\epsilon,t}(k)=\frac{1}{k}\Pd(T^{(1)}_{\epsilon,t}=k)$  for every $k\in \NN^*$ gives an explicit solution of \eqref{coageq}  with mass-conserving property on the inteval $[0;\epsilon^2]$.  Its second moment $\displaystyle{\sum_{k=1}^{+\infty}k^2\rho_{\epsilon, t}(k)=(1-\frac{t}{\epsilon^2})^{-1}}$  diverges as $t$ tends to $\epsilon^2$.  
\end{enumerate}
\end{rem}

\noindent Let us note that the set of equations 
\[
\frac{d}{dt}\rho_t(k) = G_{2}(\rho_t,k), \quad \forall k\in\NN^*
\]
corresponds to the Flory's coagulation equations with the multiplicative
kernel (see equation \eqref{Floryeq}). \\
 The following proposition shows that for every $j\geq 2$, an approximation of the solution of the set of equations  
\[
\frac{d}{dt}\rho_t(k) = G_{j}(\rho_t,k), \quad \forall k\in\NN^*
\]
can be constructed by considering only loops of length $j$:
\begin{prop}
\label{prop:hydrodynj}
 Let $j$ be an integer greater than or equal to $2$. The set of loops of length $j$ before time $t$ defines a subgraph of $\mathcal{G}_{t}^{(n)}$ denoted by $\mathcal{G}_{t}^{(n,j)}$. For $k\in \NN^*$, $n\in \NN$ and $t>0$, let $\rho_{\epsilon,t}^{(n,j)}(k)$ be the average number of
components of size $k$ in the random graph $\mathcal{G}_{nt(\epsilon+1)^j}^{(n,j)}$.
\begin{enumerate}
\item $(\rho^{(n,j)}_{\epsilon,t}(k))_n$ converges to $\rho^{(j)}_{t}(k)=e^{-tk}\dfrac{(tk)^{\frac{k-1}{j-1}}}{k^2(\frac{k-1}{j-1})!}\un_{\{k-1\in (j-1)\NN\}}$ in $L^2$ for every
$t>0$. 
\item $(\rho_{t}^{(j)}(k),\ k\in \NN^* \text{ and } t\geq 0)$ is a
solution to the following coagulation equations:
\begin{equation}
\label{coageqj}
\frac{d}{dt}\rho_t(k) =G_{j}(\rho_t,k)
\end{equation}
where $G_{j}$ is defined by equation \eqref{Gj}. 
\end{enumerate}
\end{prop}
The study of the random graph process defined by the Poisson ensemble of loops of a fixed length is postponed to Appendix \ref{section:fixedlength}. 
\section{Component exploration procedure and associated Galton-Watson process \label{sect:exploration}}
In this section, we describe a component exploration procedure modeled on the Karp \cite{Karp} and Martin-L\"of \cite{MartinLof} exploration algorithm. The aim of this  procedure is  to find $C^{(n)}_t(x)$ and to construct a  Galton-Watson process, the total population size of which bounds the size of the component $|C^{(n)}_t(x)|$. 
\subsection{Component exploration procedure} 
For every  subset of vertices $V$ let $\NL_{t}(V)$ denote the subset of
loops before time $t$ included in $V$ and  let $\NL_{t,x}(V)$ denote  the subset of those
that also pass through $x$. Let  define the set of `neighbours' of $x$ in $V$ as
$$\mathcal{N}_{t,x}(V)=\{y\in V\setminus \{x\},\; \exists \ell \in\NL_{t,x}(V)\;\text{that
passes through}\;y\}.$$ 

In each step of the algorithm, a vertex is either \emph{active}, \emph{explored}
or \emph{neutral}. Let $A_k$ and $H_k$ be the sets of active vertices and
explored vertices in step $k$ respectively. In step $0$, vertex $x_1=x$ is said
to be active ($A_0=\{x_1\}$) and other vertices are neutral. In step 1, every
neighbour is declared active and the vertex $x$ is said to be an explored
vertex: 
$A_{1}=\mathcal{N}_{t,x}(\KN)$ and  $H_1=\{x_1\}$. In step $k$, let us assume
that  $A_{k-1}$ is non-empty. Let $x_k$ denote  the smallest active vertex in
$A_{k-1}$. We add the neutral vertices $z\in \mathcal{N}_{t,x_k}(\KN\setminus
H_{k-1})$ to $A_{k-1}$ and change the status of $x_k$: $A_{k}=A_{k-1}\cup
\mathcal{N}_{t,x_k}(\KN\setminus H_{k-1})\setminus\{x_k\}$ and
$H_{k}=H_{k-1}\cup\{x_k\}$. In particular, $|A_k| = |A_{k-1}| + \xi^{(n)}_{t,k}
-1$ with $\xi^{(n)}_{t,k} = |\mathcal{N}_{t,x_k}(\KN\setminus H_{k-1})\setminus
A_{k-1}|$. The process stops in step $T^{(n)}_{t}=\min(k,\; A_k=\emptyset)$.  By
construction, $T^{(n)}_{t}=\min(k,\; \sum_{i=1}^{k}\xi^{(n)}_{t,i}\leq k-1)$,
the component of $x$ is $C^{(n)}_t(x)=H_{T^{(n)}_{t}}$ and its size is
$T^{(n)}_{t}$ (an example is presented in Figure \ref{algocomp}).
\begin{figure}[h]
\fbox{
 \begin{minipage}{0.4\textwidth}
  \includegraphics*[scale=0.5]{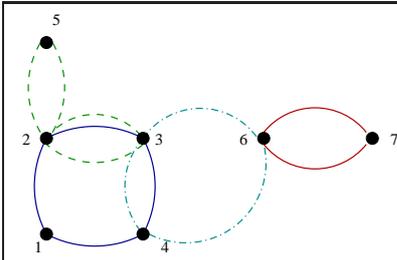}
\caption{\label{algocomp}a component of $\DL_{t}^{(n)}$ formed by four  based loops ${\ell_1=(1,2,3,4)}$, ${\ell_2=(2,5,2,3)}$, ${\ell_3=(3,6,4)}$ and  ${\ell_4=(6,7)}$. }
 \end{minipage}\quad
{\small
\begin{minipage}{0.55\textwidth}
The steps of exploration procedure for this component are
\begin{itemize}
\item Step 1: $x_{1}=1$ and $A_{1}=\{2,3,4\}$ so that $\xi^{(n)}_{t,1}=3$.
\item Step 2: $x_{2}=2$ and  $A_{2}=\{3,4,5\}$ so that $\xi^{(n)}_{t,2}=1$.
\item Step 3: $x_{3}=3$ and  $A_{3}=\{4,5,6\}$ so that $\xi^{(n)}_{t,3}=1$.
\item Step 4: $x_{4}=4$ and  $A_{4}=\{5,6\}$ so that $\xi^{(n)}_{t,4}=0$.
\item Step 5: $x_{5}=5$ and $A_{5}=\{6\}$ so that $\xi^{(n)}_{t,5}=0$.
\item Step 6: $x_{6}=6$ and  $A_{6}=\{7\}$ so that $\xi^{(n)}_{t,6}=1$.
\item Step 7: $x_{7}=7$ and  $A_{7}=\emptyset$ so that $\xi^{(n)}_{t,7}=0$.
\end{itemize}
\end{minipage}
}}
\end{figure}

\subsection{The Galton-Watson process associated to a component}
 The random variable $\xi^{(n)}_{t,k}$ is bounded above by 
$$\zeta^{(n,1)}_{t,k}=\sum_{\ell \in\NL_{t,x_k}(\KN\setminus H_{k-1})}\!\!\!(|\ell|-1)$$ 
in which a same vertex is counted as many times as it appears in loops $\ell \in \NL_{t,x_k}(\KN\setminus H_{k-1})$.  To obtain identically distributed random variables in each step, we have to consider also in step $k$,  loops that pass through $x_k$ and $H_{k-1}$ before time $t$. Let denote this set of loops  $\NL_{t,x_k,H_{k-1}}$. 
We define   $\zeta^{(n,2)}_{t,k}=\sum_{\ell
\in\NL_{t,x_k,H_{k-1}}}\!\!\!(|\ell|-1)$  and set   
$$\zeta^{(n)}_{t,k}=\zeta^{(n,1)}_{t,k}+\zeta^{(n,2)}_{t,k}=\sum_{\ell \in\NL_{t,x_k}(\KN)}(|\ell|-1).$$ 
The distribution of $\zeta^{(n)}_{t,k}$   is 
the  compound Poisson distribution
$\CPois(t\beta_{n,\epsilon},\nu_{n,\epsilon})$  with
\begin{multline*}\beta_{n,\epsilon}=\mu(\{\ell\in\DL(\KN), x\in
\ell\})=\mu(\DL(\KN))-\mu(\DL(\KN\setminus \{x\}))
=\log(1+\frac{1}{n\epsilon})-\frac{1}{n(\epsilon+1)}.
\end{multline*}
 $$\nu_{\epsilon,
n}(j)=\frac{1}{\beta_{n\epsilon}}\mu(\{\ell\in\mathcal{DL}(\KN),\
x\in\ell\;\text{and}\;|\ell|=j+1\})=\frac{1-(1-\frac{1}{n})^{j+1}}{\beta_{n,
\epsilon}(j+1)(\epsilon+1)^{j+1}}\quad \forall j\in\NN^*.$$
\begin{ex}
 For the component drawn in Figure \ref{algocomp}, the random variables associated with the 
 first three steps of the exploration procedure are $\zeta^{(n,1)}_{t,1}=3$,  
 $\zeta^{(n,2)}_{t,1}=0$, $\zeta^{(n,1)}_{t,2}=3$,  $\zeta^{(n,2)}_{t,2}=3$,  
 $\zeta^{(n,1)}_{t,3}=2$ and $\zeta^{(n,2)}_{t,3}=6$.
\end{ex}
Let $\mathcal{F}_k=\sigma(H_j, A_j,\; j\leq k)$. 
Let us note that the random variables $\zeta^{(n)}_{t,j}$ and $\zeta^{(n)}_{t,k}$ for $j<k$ 
are not independent since a same loop can belong to $\DL_{t,x_k,H_{k-1}}$ and $\DL_{t,x_j,H_{j-1}}$. 
Nevertheless, as disjoint subsets of loops in
$\DL^{(n)}_t$ are independent, 
the random variables $\zeta^{(n,1)}_{t,j}$ for $j\leq k$ are independent conditionally on 
$\mathcal{F}_{k}$, and the random variable $\zeta^{(n,1)}_{t,k}$ is
independent of  $\zeta^{(n,2)}_{t,k}$ conditionally on $\mathcal{F}_{k}$. 
 Therefore, by using independent copies of the Poisson point processes $\DL$,  
 we can construct  a sequence of nonnegative random variables  
 $(\bar{\zeta}^{(n,2)}_{t,k})_k$ such that for
every $k$:
\begin{itemize}
\item $\bar{\zeta}^{(n,2)}_{t,k}$ has the same distribution as
$\zeta^{(n,2)}_{t,k}$ and is independent of $\zeta^{(n,1)}_{t,k}$ conditionally
on $\mathcal{F}_k$. 
 \item $\bar{\zeta}^{(n)}_{t,k}=\zeta^{(n,1)}_{t,k}+\bar{\zeta}^{(n,2)}_{t,k}$
are independent with distribution 
$\CPois(\beta_{n,\epsilon}t,\nu_{n,\epsilon})$. 
\end{itemize}
Set $\bar{T}^{(n)}_t=\min(k,\
\bar{\zeta}^{(n)}_{t,1}+\ldots+\bar{\zeta}^{(n)}_{t,k}=k-1)$. By construction,
$\bar{T}^{(n)}_t\geq |C^{(n)}_{t}(x)|$. 
If $\bar{\zeta}^{(n)}_{t,1}$ is seen as the number of offspring of an individual
$I$ and $\bar{\zeta}^{(n)}_k$ for $k\geq 2$ as the number of offspring of the
$k$-th individual explored by a breadth-first algorithm of the family tree of $I$,
then $\bar{T}^{(n)}_t$ is the total number of individuals in the family tree of
$I$. We call  $(\bar{\zeta}^{(n)}_{t,k})_k$  the associated Galton-Watson
process (a bijection between  Galton-Watson trees and  lattice  walks was 
described by T.~E.~Harris~\cite{Harris52} in Section 6, 
see also Section 6.2 in \cite{PitmanLN} for a review).
\section{Approximation of component sizes\label{sect:proofth}}
The number of neighbours of a vertex  is used to approximate the number of vertices added in each 
step of the exploration process of a component. We begin this section by studying its asymptotic distribution.  
Next, we prove Theorem \ref{thmsub}. Its proof is divided into two steps:  we  give an upper bound of 
the deviation between the cumulative distribution function of 
$|C^{(n)}_{nt}(x)|$ and of the total population size of the  associated Galton-Watson process and then 
we study the asymptotic distribution of the  Galton-Watson process associated to $|C^{(n)}_{nt}(x)|$.
We end this section by a proof of Corollary \ref{coroltwocomponents}. 
\subsection{Neighbours of a vertex}
Let $V_n$ be a subset of vertices in $\bar{K}_n$ and let $x$ be another vertex. 
The aim of this section is to show that the 
number of neighbours of $x$ in $\KN\setminus V_n$ at time $nt$ (denoted by $|\Vois_{nt,x}(\KN\setminus V_n)|$)
converges in distribution to the compound Poisson distribution $\CPois(\frac{t}{\epsilon(\epsilon+1)},\Geo)$ if
$\frac{V_n}{n}$ tends to $0$. \\
The 
number of neighbours of $x$ in $\KN\setminus V_n$ at time $t$ is equal to 
$\sum_{\ell\in\DL_{t,x}(\KN\setminus V_n)}(|\ell|-1)$ except if there are loops 
in $\DL_{t,x}(\KN\setminus V_n)$ that cross a same vertex several times or that cross another 
loop of $\DL_{t,x}(\KN\setminus V_n)$ in a vertex $y\neq x$. The following lemma yields an upper
bound for the probability that such an event occurs:
\begin{lem}
\label{lemrecoupe}
Let $x$ be a vertex. Set  $G_{n,t}$ be the event  `\emph{there exists a loop in
$\NL_{t,x}(\KN)$ that crosses a same vertex several times or that
intersects another loop  in $\NL_{t,x}(\KN)$ at a vertex $y\neq x$.}'
$$
\Pd(G_{n,t})\leq \frac{t}{n^2\epsilon^3}(\epsilon+1)+\frac{t^2}{n^3\epsilon^4}. 
$$
\end{lem}
\begin{proof}
 We study separately the following two events:
\begin{itemize}
 \item 
 $G^{(1)}_{n,t}$ :`\emph{there exists a vertex $y\neq x$ which is crossed
several times by  loops in $\NL_{t,x}(\KN)$}' 
\item $G^{(2)}_{n,t}$: `\emph{there exists a loop in $\NL_{t,x}(\KN)$ that crosses $x$ several times}'.
\end{itemize}
To compute  $\Pd(G^{(1)}_{n,t})$, we introduce the random variable $S_{t,x}$
as the total length of loops in $\NL_{t,x}(\KN)$ minus the number of
times these loops pass through $x$: 
$S_{t,x}=\sum_{\ell\in \NL_{t,x}(\KN)}M_x(\ell)$ where  $M_x(\ell)$
denotes the number of vertices different from $x$ in a loop $\ell$. Since the
vertices that form a loop are chosen independently with the uniform distribution
on $\KN$,
\[\Pd(G^{(1)}_{n,t})=1-\Ed(\prod_{i=0}^{S_{t,x}-1}(1-\frac{i}{n-1}))\leq
\frac{1}{2(n-1)}\Ed(S_{t,x}(S_{t,x}-1)).\] 
 By Campbell's formula, the probability-generating function of $S_{t,x}$ is
\[\Ed( u^{S_{t,x}})=\exp\Big(\sum_{\ell\in \NL_{t,x}(\KN)}(u^
{M_x(\ell)}-1)t\mu(\ell)\Big).\]  The $\mu$-measure of loops in
$\NL_{t,x}(\KN)$ of length $j$ that cross $i$ times the vertex $x$ 
is $\displaystyle{ \binom{j}{i}\frac{(n-1)^{j-i}}{j(n(\epsilon+1))^j}}$. Using the binomial
formula, we obtain that 
$\sum_{\ell\in \NL_{t,x}(\KN)}(u^{M_x(\ell)}-1)\mu(\ell)$ is equal to:
\begin{multline*}\sum_{j=2}^{+\infty}\sum_{i=1}^{j}(u^{j-i}-1)\binom{j}{i
}\frac{(n-1)^{j-i}}{j(n(\epsilon+1))^j}\\
=\sum_{j=2}^{+\infty}\frac{1}{j(n(\epsilon+1))^j}
\Big((u(n-1)+1)^j-u^j(n-1)^j-n^j+(n-1)^j\Big)\\
=-\log(1-\frac{1}{n\epsilon})-\log(1-\frac{
u(n-1)+1}{n(\epsilon+1)})+\log(1-\frac{u(n-1)}{n(\epsilon+1)}).
\end{multline*}
Therefore, $\Ed(
u^{S_{t,x}})=(1+\frac{1}{n\epsilon})^{-t}(1-\frac{1}{n(\epsilon+1)-u(n-1)})^
{-t}$. In particular, 
$$\Ed(S_{t,x}(S_{t,x}-1))=\frac{t(n-1)^2(2n\epsilon+t+1)}{
(n\epsilon)^2(n\epsilon+1)^2}.$$ 
Thus $\Pd(G^{(1)}_{n,t,k})\leq
\frac{t(n\epsilon+t)}{n^3\epsilon^4}.$\\
To study $G^{(2)}_{n,t}$, we set  $N_x(\ell)$  the number of times a loop
$\ell$ passes through the vertex $x$. We have
 $\Pd(G^{(2)}_{n,t})=1-\exp(-t\mu(\ell\in \DL(\KN),\ N_x(\ell)\geq
2))$. 
We have already seen in Lemma \ref{lemintersect} that   
$$\mu(\ell\in \DL(\KN),\ N_x(\ell)\geq
1)=\log(1+\frac{1}{n\epsilon})-\frac{1}{n(\epsilon+1)}.$$
Finally $$\mu(\ell\in \DL(\KN),\
N_x(\ell)=1)=\sum_{j=2}^{+\infty}\frac{1}{j}\frac{j(n-1)^{j-1}}{(n(\epsilon+1))^j}
=\frac{n-1}{n(\epsilon+1)(n\epsilon+1)}=\frac{1}{n\epsilon+1}-\frac{1}{n(\epsilon+1)}.$$
Therefore, $\Pd(G^{(2)}_{n,t})\leq t\mu(\ell\in \DL(\KN),\
N_x(\ell)\geq 2)=
t\Big(\log(1+\frac{1}{n\epsilon})-\frac{1}{n\epsilon+1}\Big)\leq \frac{t}{n\epsilon(n\epsilon+1)}.$
\end{proof}
The distribution of $\sum_{\ell\in\DL_{t,x}(\KN\setminus V_n)}(|\ell|-1)$ is described in the following 
proposition: 
\begin{prop}
\label{TVlengthneig}
 Let $V_n$ be a subset of vertices and let $x$ be another vertex. 
 \begin{itemize}
 \item[(i)] The random variable $\sum_{\ell\in\DL_{x,nt}(\KN\setminus V_n)}(|\ell|-1)$
  is $\CPois(ntb_{n},\nu_{n})$-distributed where:
  $$b_n=-\frac{1}{n(\epsilon+1)}+\log(1+\frac{1}{n\epsilon+|V_n|}) \et  
  \nu_n(j)=\frac{(1-\frac{|V_n|}{n})^{j+1}-
  (1-\frac{|V_n|+1}{n})^{j+1}}{b_n(j+1)(\epsilon+1)^{j+1}}\;\forall j\in\NN^*.$$ 
\item[(ii)]  
 $\TV\left(\CPois(ntb_{n},\nu_{n}),\CPois\big(\frac{t}{\epsilon(\epsilon+1)},\Geo\big)\right)
 \leq \frac{3t}{2\epsilon^2}\Big(\frac{|V_n|}{n}+\frac{1}{2n}\Big).$ 
 \end{itemize}
 \end{prop}
\begin{proof}
\begin{itemize}
\item[(i)] By definition of the Poisson loop set, $\sum_{\ell\in\DL_{x,nt}(\KN\setminus V_n)}(|\ell|-1)$ 
has a compound Poisson distribution, $b_n=\mu(\ell\in\DL(\KN\setminus V_n),\ x\in\ell)$ and for every 
$j\in\NN^*$,
$$\nu_n(j)=\frac{1}{b_n}\mu(\ell\in\DL(\KN\setminus V_n),\ x\in\ell \et |\ell|=j+1).$$ 
\item[(ii)]
The total variation distance between two compound Poisson distributions can be bounded as follows
using coupling arguments:
\begin{lem}
\label{lem_couplingCP}
 Let $p_1$ and $p_2$ be two probability measures on $\NN$ and let $\lambda_1$
and $\lambda_2$ be two positive reals such that $\lambda_1< \lambda_2$. Then 
$$\TV(\CPois(\lambda_1,p_1),\CPois(\lambda_2,p_2))\leq
1-e^{-(\lambda_2-\lambda_1)}+\lambda_1\TV(p_1,p_2).$$
\end{lem}
\begin{proof}[Proof of Lemma \ref{lem_couplingCP}]
By Strassen's theorem, there exist two independent sequences $(X_i)_{i\in
\NN^*}$ and $(Y_i)_{i\in \NN^*}$  of \iid random variables  with distributions $p_1$ and
$p_2$ respectively such that ${\TV(p_1,p_2)=P(X_i\neq Y_i)}$ for every $i\in\NN\*$. Let
$Z_1$ and $Z_2$ be two independent Poisson-distributed random variables with
parameters $\lambda_1$ and $\lambda_2-\lambda_1$ respectively which  are
independent of the two sequences $(X_i)_i$ and $(Y_i)_i$. Set $Z=Z_1+Z_2$. Then
$$\Pd(\sum_{i=1}^{Z_1}X_i\neq \sum_{i=1}^{Z}Y_i)\leq
\Pd(Z_2>0)+\Pd(\sum_{i=1}^{Z_1}X_i\neq \sum_{i=1}^{Z_1}Y_i)$$ and
$$\Pd(\sum_{i=1}^{Z_1}X_i\neq \sum_{i=1}^{Z_1}Y_i)\leq
\sum_{k=0}^{+\infty}\Pd(Z_1=k)\sum_{i=1}^{k}\Pd(X_i\neq Y_i)=\Ed(Z_1)\TV(p_1,p_2).$$
\end{proof}
We apply Lemma \ref{lem_couplingCP} with $\lambda_1=nb_{n}t$, 
$\lambda_2=\frac{t}{\epsilon(\epsilon+1)}$, $p_1=\nu_{n}$ and
$p_2=\Geo$; 
We have: 
$nb_n\leq \frac{1}{\epsilon(\epsilon+1)}$ and      
$nb_n\geq \frac{1}{\epsilon(\epsilon+1)}-\frac{1}{\epsilon^2}\Big(\frac{|V_n|}{n}+\frac{1}{2n}\Big)$. 
By definition, 
$$\TV(p_1,p_2)=\frac{\epsilon}{2}\sum_{k=1}^{+\infty}(\epsilon+1)^{-k}|a_{k,n}|$$
with $$
a_{k,n}=1-\frac{1}{(k+1)\epsilon(\epsilon+1)b_{n}}
\Big(\big(1-\frac{|V_n|}{n}\big)^{k+1}-\big(1-\frac{|V_n|+1}{n}\big)^{k+1}\Big).$$
As, for $x\in[0,1]$ and $k\in\NN^*$, $$(1-x)^{k+1}-(1-x-\frac{1}{n})^{k+1}\leq \frac{k+1}{n}\;\et\; 
1-\frac{n}{k+1}\Big((1-x)^{k+1}-(1-x-\frac{1}{n})^{k+1}\Big)\leq k(x+\frac{1}{2n}),$$ 
 we obtain:
$$\forall k\in\NN^*,\ -\frac{1}{\epsilon^2}\Big(\frac{|V_n|}{n}+\frac{1}{2n}\Big)\leq nb_na_{k,n}
\leq \frac{k}{\epsilon(\epsilon+1)}\Big(\frac{|V_n|}{n}+\frac{1}{2n}\Big). $$
 Therefore,
\begin{alignat*}{2}
\TV(\CPois(\lambda_1,p_1),\CPois(\lambda_2,p_2))\leq &
1-\exp\Big(-\frac{t}{\epsilon^2}\big(\frac{|V_n|}{n}+\frac{1}{2n}\big)\Big)+
\frac{t}{2\epsilon^2}\Big(\frac{|V_n|}{n}+\frac{1}{2n}\Big)\\
\leq & \frac{3t}{2\epsilon^2}\Big(\frac{|V_n|}{n}+\frac{1}{2n}\Big).
\end{alignat*}
\end{itemize}
\end{proof}
In summary, it follows from Lemma \ref{lemrecoupe} and Proposition \ref{TVlengthneig} the following result
for the number of neighbours of a vertex:
\begin{prop}
\label{convvoisins}
 Let $x\in \KN$ be a vertex and let $V_n\subset \KN\setminus\{x\}$ be a subset of vertices. 
  The total variation distance between the distribution of 
 $|\mathcal{N}_{nt,x}(\KN\setminus V_n)|$ and the 
 $\CPois\big(\frac{t}{\epsilon(\epsilon+1)},\Geo\big)$ distribution 
 is smaller than 
 $$\frac{3t}{2\epsilon^2}\Big(\frac{|V_n|}{n}+\frac{1}{2n}\Big)
 +\frac{t(\epsilon+1)}{n\epsilon^3}+\frac{t^2}{n\epsilon^4}.$$
\end{prop}
\subsection{Comparison between a component size and the associated Galton-Watson
process}
The aim of this section is to prove  that  small component sizes  at time $nt$ are well approximated 
by $\bar{T}_{nt}^{(n)}$ which has the same distribution as the total population size of a 
Galton-Watson process with offspring distribution $\CPois(nt\beta_{n,\epsilon},\nu_{n,\epsilon})$ and 
a single ancestor (first step of the proof of Theorem \ref{thmsub}):
\begin{prop}\label{propcomp1}
Let $x$ be a vertex. There exist two polynomial functions $B_1$ and $B_2$  such
that for every $t>0,\epsilon>0$  and $k,n\in\NN^*$, 
 $$|\Pd(|C^{(n)}_{nt}(x)|\leq k)-\Pd(\bar{T}^{(n)}_{nt}\leq k)|\leq 
B_1(t,\frac{1}{\epsilon})\frac{k^2}{n}+B_2(t,\frac{1}{\epsilon})\frac{k^4}{n^2}
.$$
\end{prop}
Let us recall that the number of new vertices added in the $j$-th step of the exploration procedure is 
$\xi^{(n)}_{t,j}=|\Vois_{t,x_j}(\KN\setminus H_{j-1})\setminus A_{j-1}|$ where $A_{j-1}$ and 
$H_{j}=\{x_1,\ldots,x_{j-1}\}$ are respectively  the set of active 
vertices and  explored  vertices in step $j-1$. We have already seen one source of difference  
between  $\xi^{(n)}_{t,j}$ and $\zeta^{(n)}_{t,j}=\sum_{\ell\in\DL_{t,x_j}(\KN)}(|\ell|-1)$.
It is described by the  event 
\begin{quote}
$G_{n,t,j}$:  `there exists a loop in $\NL_{t,x_j}(\KN\setminus
H_{j-1})$ that crosses a same vertex several times or that crosses another loop 
in $\NL_{t,x_j}(\KN\setminus H_{j-1})$ at a vertex $y\neq x_j$'.
\end{quote}
By Lemma \ref{lemrecoupe}, the probability of this event is bounded by: 
$\frac{t}{n^2\epsilon^3}(\epsilon+1)+\frac{t^2}{n^3\epsilon^4}$. \\
There are two other sources of difference described by the following events:
\begin{itemize}
\item $\{\bar{\zeta}^{(n,2)}_{t,j}>0$\}:  `there exists a loop passing through $x_j$ and through already explored vertices $H_{j-1}$',
\item\label{eventFntj} $F_{n,t,j}$:  `there exists a loop  in $\NL_{t,x_j}(\KN\setminus
H_{j-1})$ (i.e.  passing through $x_j$ but not through $H_{j-1}$) which intersects active vertices $A_{j-1}$',
\end{itemize}
The probability of these two events can be bounded by using the following lemma:
\begin{lem}
\label{lemintersect}
 Let $A$ be a subset of vertices and let $x$ be another
vertex. For every $t>0$, 
\[\Pd(\exists \ell \in\DL_{t,x}(\KN),\ \ell \;\text{intersects}\;
A)=1-\Big(1+\frac{|A|}{n\epsilon(n\epsilon +|A|+1)}\Big)^{-t}.
\]
\end{lem}
\begin{proof}
Let $F_{A,x}$ be the subset of loops  $\ell$ which intersect  $A$ and pass through
$x$. 
$$\Pd(\exists \ell\in \DL_{t,x}(\KN), \ell \;\text{intersects}\;
A)=1-\exp(-t\mu(F_{A,x}))$$
and $$\mu(F_{A,x})=\mu(\DL_{x}(\KN))-\mu(\DL_{x}(\KN\setminus A)).$$ 
For a subset $V$ of $v$ vertices, set
$\beta_{n,\epsilon,v}=\mu(\DL_{x}(\KN\setminus V))$:
\begin{alignat*}{2}
\beta_{n,\epsilon,v}&=\mu(\DL(\KN\setminus V))-\mu(\DL(\KN\setminus
(V\cup\{x\})))\\
&=-\log(1-\frac{n-v}{n(\epsilon+1)})-\frac{n-v}{n(\epsilon+1)}+\log(1-\frac{
n-v-1}{n(\epsilon+1)})
+\frac{n-v-1}{n(\epsilon+1)}\\&=-\log(1+\frac{v}{n\epsilon})+\log(1+\frac{v+1}{
n\epsilon})-\frac{1}{n(\epsilon+1)}.
\end{alignat*}
Then,
$$
\mu(F_{A,x})=\log(1+\frac{|A|}{n\epsilon(n\epsilon+|A|+1)}).
$$
\end{proof}
With the help of these estimates, we prove Proposition \ref{propcomp1}. 
\begin{proof}[\textbf{Proof of Proposition \ref{propcomp1}}]
 As $|C^{(n)}_{nt}(x)|\leq \bar{T}^{(n)}_{nt}$, $$|\Pd(|C^{(n)}_{nt}(x)|\leq
k)-\Pd(\bar{T}^{(n)}_{nt}\leq k)|=\Pd(|C^{(n)}_{nt}(x)|\leq k\; \text{and}\;
\bar{T}^{(n)}_{nt}> k).$$ It is bounded above by $$\Pd(|C^{(n)}_{nt}(x)|\leq k\;
\text{and}\; \exists j\leq |C^{(n)}_{nt}(x)|,\
\xi^{(n)}_{nt,j}<\bar{\zeta}^{(n)}_{nt,j})\leq
\sum_{j=1}^{k}E(\un_{\{|C^{(n)}_{nt}(x)|\geq
j\}}\Pd(\xi^{(n)}_{nt,j}<\bar{\zeta}^{(n)}_{nt,j}|\mathcal{F}_{j-1})).$$ 
We have seen  that $$\Pd(\xi^{(n)}_{nt,j}<\bar{\zeta}^{(n)}_{nt,j}|\mathcal{F}_{j-1})\leq \Pd(\zeta^{(n,2)}_{nt,j}>0|\mathcal{F}_{j-1})+\Pd(F_{n,tn,j}|\mathcal{F}_{j-1})+\Pd(G_{n,tn,j}|\mathcal{F}_{j-1})$$ 
with the notations introduced page \pageref{eventFntj}. 
By Lemma  \ref{lemintersect}
\[\Pd(F_{n,tn,j}|\mathcal{F}_{j-1})\leq
\frac{t|A_{j-1}|}{n\epsilon^2}\;\text{and}\;\Pd(\zeta^{(n,2)}_{nt,j}>0|\mathcal{F}_{j-1})\leq
\frac{t(j-1)}{n\epsilon^2}
\]
and by Lemma \ref{lemrecoupe} 
\[
\Pd(G_{n,tn,j}|\mathcal{F}_{j-1})\leq
\frac{t}{n\epsilon^2}(1+\frac{1}{\epsilon}+\frac{t}{\epsilon^2}).  
\]
Therefore, 
$$\Pd(|C^{(n)}_{nt}(x)|\leq k\; \text{and}\;
\bar{T}^{(n)}_{nt}> k)\leq 
\frac{t}{n\epsilon^2}\sum_{j=1}^{k}\Big(E(|A_{j-1}|\un_{\{|C_{nt}^{(n)}(x)\geq j\}})+
(j+\frac{1}{\epsilon}+\frac{t}{\epsilon^2})\Pd(|C_{nt}^{(n)}(x)\geq j)\Big).$$
By construction $|A_{j-1}|-1=\sum_{i=1}^{j-1}(\xi^{(n)}_{nt,i}-1)$.  Let us
recall that  $\xi^{(n)}_{nt,i}$ has nonnegative integer values, it is bounded
above by $\bar{\zeta}^{(n)}_{nt,i}$ and the conditional law of
$\bar{\zeta}^{(n)}_{nt,i}$ given $\mathcal{F}_{i-1}$ is equal to the law of
$\zeta^{(n)}_{nt,1}$. Thus, 
\begin{equation}\label{m1Aj}\Ed(\un_{\{|C^{(n)}_{nt}(x)|\geq
j\}}(|A_{j-1}|-1))\leq \sum_{i=1}^{j-1}\Ed(\un_{\{|C^{(n)}_{nt}(x)|\geq
i\}}\Ed(\bar{\zeta}^{(n)}_{nt,i}|\mathcal{F}_{i-1}))\leq (j-1)E(\zeta^{(n)}_{nt,1}).
\end{equation}
To conclude we note that $(\Ed(\zeta^{(n)}_{nt,1}))_n$ converges to $\frac{t}{\epsilon^2}$, 
the expectation of 
$\CPois(\frac{t}{\epsilon(\epsilon+1)},\Geo)$ as $n$ tends to $+\infty$. Therefore, 
 there exist  positive reals $A$, $B$, $C$ such that for every $n\in\NN$, $t\geq0$ and $\epsilon>0$,  
$$|\Pd(|C^{(n)}_{nt}(x)|\leq k)-\Pd(\bar{T}^{(n)}_{nt}\leq k)|\leq
\frac{k^2t}{n\epsilon^2}(A+\frac{B}{\epsilon}+\frac{Ct}{\epsilon^2}).$$  
\end{proof}

\subsection{The total progeny of the Galton-Watson process associated to a component}
 Recall that the offspring distribution of the Galton-Watson process associated to a 
 component at time $nt$ is the compound
Poisson distribution $\CPois(tn\beta_{n,\epsilon},\nu_{n,\epsilon})$ with:
$$\beta_{n,\epsilon}=\log(1+\frac{1}{n\epsilon})-\frac{1}{n(\epsilon+1)}.$$
and 
 $$\nu_{\epsilon,
n}(j)=\frac{1-(1-\frac{1}{n})^{j+1}}{\beta_{n,
\epsilon}(j+1)(\epsilon+1)^{j+1}}\quad \forall j\in\NN^*.$$
We have shown (Proposition \ref{TVlengthneig}) that the compound Poisson
distribution $\CPois(tn\beta_{n,\epsilon},\nu_{n,\epsilon})$ is close to 
the $\CPois(\frac{t}{\epsilon(\epsilon+1)},\Geo)$-distribution for large $n$.
We now consider the distribution of  the total number of individuals in a
Galton-Watson process with one ancestor and  offspring distribution
$\CPois(tn\beta_{n,\epsilon}, \nu_{n,\epsilon})$. 
Let us state a general result for the comparison of the total number of individuals in two
Galton-Watson process: 
\begin{lem}
\label{lem_couplingT}
 Let $\nu_1$ and $\nu_2$ be two probability distributions on $\NN$. 
 Let $\TV$ denote the total variation distance between probability measures. Let $T_1$
and $T_2$ be the total population sizes of the Galton-Watson processes with one ancestor and
offspring distribution $\nu_1$ and $\nu_2$ respectively. \\
For every $k\in \NN^*$, $|\Pd(T_1\geq k)-\Pd(T_2\geq k)|\leq
2\TV(\nu_1,\nu_2)\sum_{i=1}^{k-1}\Pd(T_2\geq i)$. 
\end{lem}
\begin{proof}
We follow the proof of Theorem 3.20  in \cite{bookvanderHofstad} which  states
an analogous result between binomial and Poisson branching processes.
The proof is based on the description of the total population size by means of 
the hitting time of a random walk  and  coupling arguments. 
 By Strassen's theorem, there exist two independent sequences $(X_i)_{i\in
\NN^*}$ and $(Y_i)_{i\in \NN^*}$  of \iid random variables with distributions $\nu_1$
and $\nu_2$ respectively such that ${\TV(\nu_1,\nu_2)=P(X_i\neq Y_i)}$ for every
$i\in\NN\*$. Let $\tau_1=\min(n,\ X_1+ \ldots +X_n=n-1)$ and $\tau_2=\min(n,\
Y_1+ \ldots +Y_n=n-1)$.  $\tau_1$ and $\tau_2$ have the same law as $T_1$ and
$T_2$ respectively. Let $k\in\NN^*$. 
\begin{alignat*}{2}\Pd(\tau_1\geq k\; \text{and}\; \tau_2<k)&=\Pd(\exists i\leq
k-1,\ Y_i\neq X_i\; \text{and}\; \tau_1\geq k)\\ &\leq
\sum_{i=1}^{k-1}\Pd(X_j=Y_j\;\forall j\leq i-1, \; X_i\neq Y_i \; \text{and}\;
\tau_1\geq k).
\end{alignat*}
 Since $\{X_j=Y_j\;\forall j\leq i-1\; \text{and}\; \tau_1\geq k\}\subset
\{\tau_2\geq i\}$ and since $\{\tau_2\geq i\}$ depends only on $Y_1,\ldots,Y_{i-1}$,
we obtain:
$$\Pd(\tau_1\geq k\; \text{and}\; \tau_2<k)\leq \sum_{i=1}^{k-1}\Pd(\tau_2\geq
i)\Pd(X_i\neq Y_i)=\TV(\nu_1,\nu_2)\sum_{i=1}^{k-1}\Pd(\tau_2\geq i). $$
Similarly, 
$$\Pd(\tau_1<k\;\text{and}\;\tau_2\geq k)\leq \sum_{i=1}^{k-1}\Pd(X_i\neq Y_i\;\text{and}\; \tau_2\geq i)
\leq \TV(\nu_1,\nu_2)\sum_{i=1}^{k-1}\Pd(\tau_2\geq i).$$
\end{proof}
From Lemma \ref{lem_couplingT} and Proposition \ref{TVlengthneig}, we obtain: 
\begin{prop}
\label{propcouplingtotalpopul}
Let $t>0$ and $n\in \NN^*$. Let $T^{(n)}_{t}$ and $T_{t}$ denote  the total number of individuals in a
Galton-Watson process with one ancestor and  offspring distribution
$\CPois(tn\beta_{n,\epsilon}, \nu_{n,\epsilon})$  and 
$\CPois(\frac{t}{\epsilon(\epsilon+1)},\mathcal{G}_{\NN^*}(\frac{\epsilon}{
\epsilon+1}))$ respectively. 
$$|\Pd(T^{(n)}_{nt}\geq k)-\Pd(T_{t}\geq
k)|\leq \frac{3t(k-1)}{2n\epsilon^2}\quad\text{for every}\; k\in\NN^*$$ 
\end{prop}
Theorem \ref{thmsub} follows from Propositions \ref{propcomp1} and \ref{propcouplingtotalpopul}. 
\subsection{Asymptotic distribution of two component sizes}
Theorem \ref{thmsub} can be extended to a joint limit theorem for the sizes of  two components
following the proof used in \cite{Bertoinbookfrag} for  the Erd\"os-R\'enyi
random graph process:
\begin{corol}
\label{coroltwocomponents}
Let $x$ and $y$ be two distinct vertices of $\bar{K}_n$. 
For every $t>0$, $j,k\in\NN^*$, $\Pd(|C^{(n)}_{nt}(x)|=j \et |C^{(n)}_{nt}(y)|=k)$ converges to $\Pd(T^{(1)}_{\epsilon,t}=i)\Pd( T^{(1)}_{\epsilon,t}=j)$ as $n$ tends to $+\infty$. 
\end{corol}

\begin{proof}[Proof of Corollary \ref{coroltwocomponents}]
 The proof is similar to the proof presented in \cite{Bertoinbookfrag} for the
Erd\"os-R\'enyi random graph. It is based on the following properties: 
\begin{itemize}
 \item[(i)] 
the vertices of $\bar{K}_n$ play the same role,
\item[(ii)]  for every subset $A$ of $\KN$, the loop set inside $A$ at time $t$,
$\DL^{(n)}_{t}(A)$ is  associated with the restriction  of $\mu$ to $A$ (denoted
$\mu^{(A)}$) and is independent of $\DL^{(n)}_{t}\setminus\DL_{t}^{(n)}(A)$. The measure
$\mu^{(A)}$ can be seen as the loop measure on $\bar{K}_n$  with vertex set
$A$, unit conductances and killing measure
$\kappa_{n,A}=|A|\epsilon_{n,|A|}=n-|A|+n\epsilon$.  Let us note that 
$\epsilon_{n,|A|}=\epsilon+\frac{n-|A|}{|A|}(\epsilon+1)$ converges to
$\epsilon$ if $|A|/n$ tends to 1.
\end{itemize}
Let $x$ and $y$ be two distinct vertices and let $j,k$ be two nonnegative
integers.  
 We have to study the convergence of $\Pd(|C^{(n)}_{\epsilon,nt}(x)|=j \et 
|C^{(n)}_{\epsilon,nt}(y)|=k)$. 
First, let us note that by (i), for every $n\geq j$, $\Pd(y\in
C^{(n)}_{\epsilon,nt}(x) \mid |C^{(n)}_{\epsilon,nt}(x)|=j)= \frac{j-1}{n-1}$.
Therefore, $\Pd(y\in C^{(n)}_{\epsilon,nt}(x) \et |C^{(n)}_{\epsilon,nt}(x)|=j)$
converges to $0$ as $n$ tends to $+\infty$. \\
By (ii),  $\Pd(|C^{(n)}_{\epsilon,nt}(y)|=k \mid y\not \in
C^{(n)}_{\epsilon,nt}(x) \et
|C^{(n)}_{\epsilon,nt}(x)|=j)=\Pd(|C^{(n-j)}_{\epsilon_{n,n-j},nt}(y)|=k ).$ 
The convergence result stated in Theorem \ref{thmsub} still holds if $t$ is replaced  by 
the sequence $(t_n)_n$  defined by ${t_n=(1-\frac{j}{n})t}$ $\forall n\in \NN^*$,  and 
if $\epsilon$ is replaced by the sequence $(\epsilon_{n,n-j})_n$ which converges to $\epsilon$.  
Therefore, $\Pd(y\not \in C^{(n)}_{\epsilon,nt}(x),\
|C^{(n)}_{\epsilon,nt}(x)|=j \et  |C^{(n)}_{\epsilon,nt}(y)|=k)$ which is equal
to 
\begin{multline*}\Pd(|C^{(n)}_{\epsilon,nt}(y)|=k \mid y\not \in
C^{(n)}_{\epsilon,nt}(x) \et |C^{(n)}_{\epsilon,nt}(x)|=j)\\ \times\big(1-\Pd(y\in
C^{(n)}_{\epsilon,nt}(x) \mid
|C^{(n)}_{\epsilon,nt}(x)|=j)\big)\Pd(|C^{(n)}_{\epsilon,nt}(x)|=j).
 \end{multline*}
converges to $\Pd(T^{(1)}_{\epsilon,t}=k)\Pd(T^{(1)}_{\epsilon,t}=j)$. 
\end{proof}
\section{Phase transition\label{sect:transition}}
This section is devoted to the proof of Theorem \ref{thmtransition}. 
The expectation of the compound Poisson distribution $\CPois(\frac{t}{\epsilon(\epsilon+1)},\Geo)$
is $\frac{t}{\epsilon^2}$. Thus the limiting Galton-Watson process associated to a component 
is subcritical or supercritical depending on whether $t$ is smaller or larger than $\epsilon^2$. 
\subsection{The subcritical regime}
An application of the component exploration procedure and a Chernov bound allow to prove 
that when $t< \epsilon^2$, the largest component size at time $nt$ is at most of order 
$\log(n)$ with probability that converges to $1$: 
\begin{bibli}[\ref{thmtransition}.(i)]
Let $0<t<\epsilon^2$. Set
$h(t)=\sup_{\theta\in]0,\log(\epsilon+1)[}(\theta-\log(L_{t,\epsilon}(\theta)))$
where $L_{t,\epsilon}$ is moment-generating function of the compound Poisson
distribution  $\CPois(\frac{t}{\epsilon(\epsilon+1)},\Geo)$.\footnote{$h(t)$ is
the value of the Cram\'er function at 1 of
$\CPois(\frac{t}{\epsilon(\epsilon+1)},\Geo)$. As the expectation of
$\CPois(\frac{t}{\epsilon(\epsilon+1)},\Geo)$ is $\frac{t}{\epsilon^2}$, $h(t)$ 
is positive for $t< \epsilon^2$ and vanishes at $t=\epsilon^2$.} \\
For every $a>(h(t))^{-1}$, 
$\Pd(\max_{x\in \KN}|C_{nt}^{(n)}(x)|> a\log(n))$ converges to $0$ as $n$ tends
to $+\infty$. 
\end{bibli}
\begin{proof}
Let $k\in \NN^*$. 
By construction of the random variables $\xi^{(n)}_{t,j}$ and
$\bar{\zeta}^{(n)}_{t,j}$,  $$\Pd(|C^{(n)}_{t}(v)|> k)\leq
\Pd(\sum_{i=1}^{k}\xi^{(n)}_{t,i}\geq k)\leq
\Pd(\sum_{i=1}^{k}\bar{\zeta}^{(n)}_{t,i}\geq k).$$ 
 The moment-generating function  of $\bar{\zeta}_{nt,x}^{(n)}$ is finite on 
$[0,\log(\epsilon+1)[$ and is equal to $$\Ed(e^{\theta
\bar{\zeta}_{nt,x}^{(n)}})=\exp\Big(-nt\big(\log(1+\frac{1}{n\epsilon})-e^{
-\theta}\log(1+\frac{e^{\theta}}{n(\epsilon+1-e^{\theta})})\big)\Big).$$ The
moment-generating function of the compound Poisson distribution
$\CPois(\frac{t}{\epsilon(\epsilon+1)},\Geo)$ is
$L_{t,\epsilon}(\theta)=\exp(-\frac{t}{\epsilon}+\frac{t}{\epsilon+1-e^{\theta}}
)$ for every $\theta\in[0,\log(1+\epsilon)[$. 
Therefore, 
$\Ed(e^{\theta\bar{\zeta}_{nt,x}^{(n)}})=
L_{t,\epsilon}(\theta)\exp(tg_n(\theta))$ where
$$g_n(\theta)=\sum_{j=1}^{+\infty}\frac{e^{\theta j}-1}{(j+1)(\epsilon+1)^{j+1}}
\Big(1-(1-\frac{1}{n})^{j+1}-\frac{j+1}{n}\Big)\leq 0\;\text{for}\; \theta\geq 0.$$ 
 By Markov's inequality:
$$\Pd(|C^{(n)}_{nt}(x)|> k)\leq
\Ed(e^{\theta\bar{\zeta}_{nt,x}^{(n)}})^ke^{-k\theta}\leq
\exp\Big(-k\big(\theta-\log(L_{t,\epsilon}(\theta))
\big)\Big)\quad \forall 0<\theta<\log(\epsilon+1). 
$$ 
We deduce that for every $ k\in \NN^*$, $\Pd(|C^{(n)}_{nt}(x)|> k)\leq
\exp(-kh(t))$. 
In particular, for every $a>0$,  $\Pd(\max_{v\in[n]} |C^{(n)}_{nt}(v)|>
a\log(n))\leq n^{1-ah(t)}\exp(h(t))$ which completes the
proof. 
\end{proof}
\subsection{The supercritical regime}
When $t>\epsilon^2$,  the Galton-Watson process with 
family size distribution $\CPois(\frac{t}{\epsilon(\epsilon+1)},\Geo)$ is supercritical. 
Let  $q_{t,\epsilon}$ be the extinction probability 
of this Galton-Watson process starting with one ancestor. \\
We  show that there is a constant $c>0$ such that with high probability there is only one 
component with more than $c\log(n)$ vertices, and the size of this component is equivalent 
to $n(1-q_{t,\epsilon})$: 
\begin{bibli}[\ref{thmtransition}.(ii)]
\label{propgiantcomponent}
 Let $C^{(n)}_{nt,m_1}$ and $C^{(n)}_{nt,m_2}$ denote the first and second largest components of 
 the random graph $\mathcal{G}^{(n)}_{nt}$. Assume that $t>\epsilon^2$. \\
For every $a\in]1/2,1[$, there exist $\delta>0$ and $c>0$ such that 
$$\Pd(\mid |C^{(n)}_{nt,m_1}|-n(1-q_{t,\epsilon})\mid \geq n^a)+
\Pd(|C^{(n)}_{t,m_2}|\geq c\log(n))=O(n^{-\delta}).$$
\end{bibli}
The proof consists of four steps: 
\begin{enumerate}
 \item In the first step, we show that a vertex has a component of size greater than $c\log(n)$ with a probability equivalent to 
 the Galton-Watson process survival probability $1-q_{t,\epsilon}$. 
 \begin{prop}
 \label{surcritcompsize}
  Let $X$ denote a $\CPois(\frac{t}{\epsilon(\epsilon+1)},\Geo)$-distributed random variable with $t>\epsilon^2$. 
  Set $I_t=\sup_{\theta\geq 0}(-\log  \Ed(e^{-\theta X})-\theta)$. \\
  For every $a>I_{t}^{-1}$, $\Pd(|C^{(n)}_{nt}(v)|\geq a\log(n))=1-q_{t,\epsilon}+O(\frac{\log^2(n)}{n})$. 
 \end{prop}
 \item For $k\in\NN$, let $Z_{nt}(k)$ denote the number of vertices that belong to a component 
 of size greater than or equal to $k$  at time $nt$. 
In the second step,  we study the first two moments of $Z_{nt}(k)$  in order to prove:
 \begin{prop}
 \label{surcritnblargecomp}
  For every $b\in]1/2;1[$, there exists $\delta>0$ such that if $a>I_{t}^{-1}$ then 
  $\Pd(|Z_{nt}(a\log(n))-n(1-q_{t,\epsilon})| > n^{b})=O(n^{-\delta})$.
 \end{prop}
 \item The aim of the third step is to prove that with high probability, there is no component 
 of size between $c_1\log(n)$ and $c_2n^\beta$ for any constant $\beta\in]1/2,1[$. More precisely, 
 we show the following result on $A_{k}(v)$, the set of active vertices  in step $k$ of the 
 exploration
 of the component of a vertex $v$:
 \begin{prop}
  \label{propintermedsize}
  Let $\beta\in]1/2,1[$. For every $0<c_2<\min(1,\frac{t}{\epsilon^2}-1)$, there exists $\delta(c_2)>0$ such that for 
  $c_1>\delta^{-1}(c_2)$, 
  $$\Pd\big(\exists v\in\KN,\ A_{c_1\log(n)}(v)\neq \emptyset\text{ and } 
  \exists k\in[c_1\log(n),n^{\beta}],\ |A_k(v)|\leq c_2k\big)=O(n^{1-c_1\delta(c_2)}).$$
 \end{prop}
\item In the fourth step, we deduce from Proposition \ref{propintermedsize} that with high 
probability there  exists at most one  component of size greater than $a\log(n)$:
\begin{prop}
  \label{uniquegiantcomp}
  For every $0<c_2<\min(1,\frac{t}{\epsilon^2}-1)$, there exists $\delta(c_2)>0$ such that 
  for $c_1>\delta^{-1}(c_2)$, 
  $$\Pd\big(\text{there exist two distinct components of size greater than}\; 
  c_1\log(n)\big)=O(n^{1-c_1\delta(c_2)}).$$
 \end{prop}
  \end{enumerate}
 Assertion (ii) of Theorem \ref{thmtransition} is then a direct consequence of Proposition
\ref{surcritnblargecomp}  and Proposition \ref{uniquegiantcomp} since 
$Z_{nt}(c_1\log(n))$ is equal to the size of the largest component on the event: 
$$\{|Z_{nt}(c_1\log(n))-n(1-q_{t,\epsilon})|\leq n^{b}\}\cap 
\{\text{there is at most one component of size greater than}\;c_1\log(n)\}.$$
 The first two steps of the proof of assertion (ii) of Theorem \ref{thmtransition} are similar to the 
 first two steps detailed in \cite{bookvanderHofstad} for 
 the Erd\"os-R\'enyi random graph. The last two steps follow the proof described in 
 \cite{BordenaveNotes} for the Erd\"os-R\'enyi random graph.  
 \begin{proof}[Proof of Proposition \ref{surcritcompsize}]
  Let $v$ be a vertex. By Theorem \ref{thmsub}, for every $a>0$, 
  $$\Pd(|C^{(n)}_{nt}(v)|\geq c\log(n))=\Pd(T^{(1)}_{\epsilon,t}\geq c\log(n))+O(\frac{\log^2(n)}{n}).$$
  Moreover, $\Pd(T^{(1)}_{\epsilon,t}=+\infty)=1-q_{t,\epsilon}$.
  To complete the proof, we use the following result on the total progeny of a 
  supercritical Galton-Watson process stated in \cite{bookvanderHofstad}: 
  \begin{bibli}[3.8 in \cite{bookvanderHofstad}]
   Let $T$ denote the total progeny of a Galton-Watson process with family size distribution 
   $\nu$. 
   Assume that $\sum_{k\in\NN}k\nu(k)>1$. \\
   Then $I=\sup_{\theta\geq 0}\Big(-\theta-\log(\sum_{k=0}^{+\infty} e^{-\theta x}\nu(k))\Big)$ 
   is positive and 
   $\Pd(k\leq T<+\infty)\leq \frac{e^{-kI}}{1-e^{-I}}$.  
  \end{bibli}
  Therefore, for every $c>I^{-1}$, $\Pd(c\log(n)\leq T^{(1)}_{\epsilon,t}<+\infty)=O(n^{-1})$ and 
  $$\Pd(|C^{(n)}_{nt}(v)|\geq c\log(n))=1-q_{t,\epsilon}+O(\frac{\log^2(n)}{n}).$$
 \end{proof}
\begin{proof}[Proof of Proposition \ref{surcritnblargecomp}]
We shall use Bienaym\'e-Chebyshev inequality to bound  
$$\Pd(|Z_{nt}(a\log(n))-n(1-q_{t,\epsilon})| > n^{b}).$$ 
As $Z_{nt}(k)=\sum_{x\in\KN}\un_{\{|C^{(n)}_{nt}(x)|\geq k\}}$, 
we deduce from Proposition \ref{surcritcompsize} that 
if  $a>I_{t}^{-1}$ then $$\Ed(Z_{nt}(a\log(n)))=-n(1-q_{t,\epsilon})+O(\log^2(n)).$$ 
We  proceed as in \cite{bookvanderHofstad} to bound the variance of $Z_{nt}(k)$. The computations
are based on the properties (i) and (ii) of $\mathcal{G}^{(n)}_{t}$ stated in the proof of 
Corollary \ref{coroltwocomponents} which implies that 
$$\Pd(|C^{(n)}_{nt}(y)|<k \mid y\not\in C^{(n)}_{nt}(x) \et |C^{(n)}_{nt}(x)|=h)-\Pd(|C^{(n)}_{nt}(y)|<k)$$
is bounded above by 
the probability that there exist loops $\ell\in \DL_{nt}^{(n)}$ passing through the two subsets 
of vertices $\{1,\ldots,k\}$ and $\{k+1,\ldots,k+h\}$. \\
We now explain in detail the computations. The variance of $Z_{nt}(k)$ is equal to the variance of  
$\sum_{x\in\KN}\un_{\{|C^{(n)}_{nt}(x)|< k\}}$. Therefore, 
$$\var(Z_{nt}(k))=\sum_{x,y\in\KN}\Big(\Pd(|C^{(n)}_{nt}(x)|< k\;\text{and}\;|C^{(n)}_{nt}(y)|< k)-
\Pd(|C^{(n)}_{nt}(x)|< k)\Pd(|C^{(n)}_{nt}(y)|< k)\Big).$$
We split $\Pd(|C^{(n)}_{nt}(x)|< k\;\text{and}\;|C^{(n)}_{nt}(y)|< k)$ into two terms depending 
on whether the vertices $x$ and $y$ belong to a same component or not:  
$\var(Z_{nt}(k))=S^{(1)}_{n}(k)+S^{(2)}_{n}(k)$ where 
\begin{alignat*}{2}
S^{(1)}_{n}(k)=&\sum_{x,y\in\KN}\Pd\big[|C^{(n)}_{nt}(x)|< k\;\text{and}\; y\in C^{(n)}_{nt}(x)\big]\\
S^{(2)}_{n}(k)=&\sum_{x,y\in\KN}\Big(\Pd\big[|C^{(n)}_{nt}(x)|< k,\ |C^{(n)}_{nt}(y)|< k \et 
y\not\in C^{(n)}_{nt}(x)\big]-
\Pd\big[|C^{(n)}_{nt}(x)|< k\big]\Pd\big[|C^{(n)}_{nt}(y)|< k\big]\Big).
\end{alignat*}
First,  $S^{(1)}_{n}(k)=n\Ed(|C^{(n)}_{nt}(1)|\un_{\{|C^{(n)}_{nt}(1)|<k\}})\leq nk. $\\
We consider now the following term in $S^{(2)}_{n}(k)$: 
\[\Pd\big[|C^{(n)}_{nt}(x)|< k,\ |C^{(n)}_{nt}(y)|< k \et y\not\in C^{(n)}_{nt}(x)\big]
 =\sum_{h=1}^{k-1}\Pd\big[|C^{(n)}_{nt}(x)|=h,\ |C^{(n)}_{nt}(y)|< k \et y\not\in C^{(n)}_{nt}(x)\big].
\]  
 For an integer 
$h< k$, 
\begin{multline*}
\Pd\big[|C^{(n)}_{nt}(x)|= h,\ |C^{(n)}_{nt}(y)|< k \et y\not\in C^{(n)}_{nt}(x)\big]\\
\leq \Pd\big[|C^{(n)}_{nt}(x)|= h\big]
\Pd\big[|C^{(n)}_{nt}(y)|< k \mid  y\not\in C^{(n)}_{nt}(x) \et |C^{(n)}_{nt}(x)|= h\big].
\end{multline*}
Let $\mathcal{G}^{(n,h)}_{nt}$ denote 
the random graph generated 
by the loops included in the subset of vertices $\KN[n-h]$ and 
let $C^{(n,h)}_{nt}(1)$ denote the component of  the vertex $1$ in $\mathcal{G}^{(n,h)}_{nt}$. 
By the properties of the Poisson loop ensemble, 
$$\Pd\big[|C^{(n)}_{nt}(y)|< k \mid  y\not\in C^{(n)}_{nt}(x) \et |C^{(n)}_{nt}(x)|= h\big]=
\Pd\big[|C^{(n,h)}_{nt}(1)|< k \big].$$ 
We can couple  $\mathcal{G}^{(n,h)}_{nt}$ and $\mathcal{G}^{(n)}_{nt}$ 
by adding to  $\mathcal{G}^{(n,h)}_{nt}$, $h$ vertices and the loops of an independent 
Poisson point process on $\RR^{+}\otimes \DL(\KN)$ at time $nt$ that are not included in $\KN[n-h]$. 
Therefore, 
$\Pd(|C^{(n,h)}_{nt}(1)|< k )-\Pd(|C^{(n)}_{nt}(1)|< k )$ is equal to the probability that 
the component of the vertex $1$  in $\mathcal{G}^{(n,h)}_{nt}$ is smaller than $k$ 
and that the component of $1$ in $\mathcal{G}^{(n)}_{nt}$ is greater or equal 
to $k$. This probability is bounded above by the 
probability that there exist  loops $\ell\in \DL_{nt}^{(n)}$ passing through the two subsets 
of vertices $\{1,\ldots,k\}$ and $\{n-h+1,\ldots,n\}$.
Therefore, 
\begin{multline*}
\Pd(|C^{(n,h)}_{nt}(1)|< k )-\Pd(|C^{(n)}_{nt}(1)|< k )\\\leq 
1-\exp\Big(-nt\big(\mu(\ell\in \DL(\KN),\ \ell \inter \KN[k]\et  \{k+1,\ldots,k+h\}\big)\Big)\\
=1-\Big(1+\frac{kh}{n\epsilon(k+h+n\epsilon)}\Big)^{-nt}
\leq 1-\Big(1+\frac{k^2}{n^2\epsilon^2}\Big)^{-nt}.
\end{multline*}
We deduce that 
$S^{(2)}_{n}(k)\leq n^2\Pd(|C^{(n)}_{nt}(1)|<k)
\Big(1-\big(1+\frac{k^2}{n^2\epsilon^2}\big)^{-nt}\Big)$
and  
$$\var(Z_{nt}(k))\leq nk+n^2\Big(1-\big(1+\frac{k^2}{n^2\epsilon^2}\big)^{-nt}\Big).$$ 
Let us note that for every $\delta>0$, $\dfrac{\var(Z_{nt}(a\log(n)))}{n^{1+\delta}}$
 converges to $0$ as $n$ tends to $+\infty$. Therefore,  
 Bienaym\'e-Chebyshev inequality is sufficient
 to complete the proof. 
\end{proof}
\begin{proof}[Proof of Proposition \ref{propintermedsize}]
Let $\beta\in]1/2,1[$. 
The idea of the proof is to lower bound the number of new active vertices at the first steps 
of the component exploration procedure by considering only loops inside a subset of 
$m_n=n-\lceil2n^\beta\rceil$ vertices. For large $n$, the Galton-Watson associated to this component
exploration procedure is still supercritical. \\
Let $\tau=T^{(n)}_t\wedge \min(k\in\NN^*,\ \sum_{i=1}^{k}\xi^{(n)}_{t,i}\geq 2n^{\beta})$.
 On the event $\{k\leq \tau\}$, the number of neutral sites at step $k$ is greater than $m_n$. 
 Let $U_k$ denote the set of the $m_n$ first neutral vertices at step $k$ and 
 let $Y^{(n)}_{t,k+1}$ denote the number of vertices $y\in U_k$ which are crossed by a loop 
 $\ell\in \DL_{t,x_k}(U_k\cup \{x_k\})$.
 On the event $\{k\leq \tau\}$, $Y^{(n)}_{t,k+1}\leq \xi^{(n)}_{t,k+1}$. Therefore, 
 $\sum_{i=1}^{k\wedge \tau}Y^{(n)}_{t,i}\leq \sum_{i=1}^{k\wedge \tau}\xi^{(n)}_{t,i}$.
\\
For a vertex $v$, set  
$$\Omega^{(n)}_{c_1,c_2}(v)=\{A_{c_1\log(n)}(v)\neq \emptyset\text{ and } 
  \exists k\in[c_1\log(n),n^{\beta}],\ |A_k(v)|\leq c_2k\}. $$
  On the event $\{k\leq \tau \et  |A_k(v)|\leq c_2k\}$, 
  $\sum_{i=1}^{k}Y^{(n)}_{t,i}$ is bounded above by $(c_2+1)k-1$. 
  Thus, 
\begin{eqnarray*}\Pd(\Omega^{(n)}_{c_1,c_2}(v))&\leq &\sum_{k=c_1\log(n)}^{n^{\beta}}
\Ed\left(\Pd( A_{c_1\log(n)}(v)\neq \emptyset\text{ and }  
|A_k(v)|\leq c_2k \mid \mathcal{F}_{k-1})\un_{\{k\leq \tau\}}\right)\\
&\leq& \sum_{k=c_1\log(n)}^{n^{\beta}}\Pd\left(\sum_{i=1}^{k}\tilde{Y}^{(n)}_{t,i}\leq (c_2+1)k-1\right).
\end{eqnarray*}
where $(\tilde{Y}^{(n)}_{t,i})_i$ denotes a sequence of independent random variables distributed 
as $|\mathcal{N}_{t,1}(\KN[m_n+1])|$. 
The last step consists in  establishing an exponential bound for
$$p_{n,k}:=\Pd\Big(\sum_{i=1}^{k}\tilde{Y}^{(n)}_{t,i}\leq (c_2+1)k-1\Big)$$ uniformly on $n$. A such exponential 
bound is an easy consequence of the following two facts:
\begin{itemize}
\item[(i)]  $c_2+1$ is smaller than the expectation of the 
$\CPois(\frac{t}{\epsilon(\epsilon+1)},\Geo)$ distribution. 
\item[(ii)] $(\tilde{Y}^{(n)}_{t,1})_n$ converges in law to the $\CPois(\frac{t}{\epsilon(\epsilon+1)},\Geo)$ 
distribution (Proposition \ref{convvoisins}). 
\end{itemize}
For every $\theta>0$, $p_{n,k}\leq \exp(k\Lambda_n(-\theta))$ where 
$\Lambda_n(\theta)=\log\Big(\Ed(e^{\theta(\tilde{Y}^{(n)}_{t,1}-(c_2+1))})\Big)$. 
Let $Y$ be $\CPois(\frac{t}{\epsilon(\epsilon+1)},\Geo)$-distributed random variable.  
Set  $\Lambda(\theta)=\log\big(\Ed(e^{\theta Y-(c_2+1))})\big)$  
for $\theta<\log(1+\epsilon)$. As $c_2+1<\Ed(Y)$,  $\Lambda'(0)>0$ and thus there exists $u^*<0$ such
that $\Lambda(u^*)<0$. Set $\delta=\frac12\Lambda(u^*)$. 
By assertion (ii),  $\Lambda_n(u^*)$ converges to $\Lambda(u^*)$, hence there exists $n^*$ such that 
for every $n\geq n^*$ and $k\in\NN^*$, $p_{n,k}\leq \exp(-k\delta)$. 
We deduce that  for  $n\geq n^*$,
$$\Pd\Big(\underset{v\in\KN}{\cup}\Omega^{(n)}_{c_1,c_2}(v)\Big)\leq n\Pd(\Omega^{(n)}_{c_1,c_2}(1))\leq n^{1-c_1\delta}(1-e^{-\delta})^{-1}$$
 which converges to 0 if  $c_1>\delta^{-1}$. \\
\end{proof}
\begin{proof}[Proof of Proposition \ref{uniquegiantcomp}]
Let  $\Omega^{(n)}_{c_1,c_2}$ denote the event 
$$\{\exists  x\in\KN\;\text{such that}\;
A_{c_1\log(n)}(x)\neq \emptyset \et \exists k\in[c_1\log(n), n^{\beta}]\;\text{such that}\;
|A_k(x)|\leq c_2k\}.$$
 It occurs  with probability $O(n^{1-c_1\delta(c_2)})$ by Proposition \ref{propintermedsize}. \\
Assume that $\Omega^{(n)}_{c_1,c_2}$ does not hold and  that there exist two vertices $x_1$ and $x_2$ 
the components of which are different and are
both of size greater than $c_1\log(n)$. 
The subsets of active vertices in step $n^{\beta}$, $A_{n^\beta}(x_1)$ and $A_{n^\beta}(x_2)$, are disjoint 
and both of size greater than 
$c_2n^\beta$. It means that no loop $\ell\in\DL^{(n)}_{nt}$  passes through $A_{n^\beta}(x_1)$ and 
$A_{n^\beta}(x_2)$. 
Note that if $F_1$ and $F_2$ are two disjoint subsets of vertices then 
\begin{multline*}\Pd(\nexists\ell \in\DL^{(n)}_{nt},\; \ell \inter F_1 \et F_2)=
\exp\Big(-nt\mu(\ell \in\DL(\KN),\; \ell \inter F_1 \et F_2)\Big)
\\=\Big(1-\frac{\frac{|F_1||F_2|}{n^2\epsilon^2}}{(1+\frac{|F_1|}{n\epsilon})(1+\frac{|F_2|}{n\epsilon})}\Big)^{nt} 
 \leq \exp\Big(-\frac{t|F_1||F_2|}{n(\epsilon+1)^2}\Big).
 \end{multline*}
Therefore there exists two different components of size greater than $c_1\log(n)$ 
with a probability smaller than the sum of $\Pd(\Omega^{(n)}_{c_1,c_2})$ and 
\begin{multline*}
\Ed\Big(\sum_{\substack{x_1,x_2\in \KN,\ A_{n^{\beta}}(x_1)\cap A_{n^{\beta}}(x_2)= \emptyset\\ 
|A_{n^{\beta}}(x_1)| > c_1\log(n),\ |A_{n^{\beta}}(x_2)| > c_1\log(n)}}\Pd(\nexists\ell \in\DL^{(n)}_{nt},\ \ell \inter A_{n^{\beta}}(x_2)
\et A_{n^{\beta}}(x_2) \mid \mathcal{F}_{n^\beta})\Big)\\
\leq n^2\exp\Big(-\frac{tc_{2}^{2}n^{2\beta-1}}{(\epsilon+1)^2}\Big). 
\end{multline*}
\end{proof}

\section{Hydrodynamic behavior of the coalescent process\label{sect:coageq}}
This section is devoted to the proof of Proposition  \ref{prop:hydrodyn}. 
\begin{enumerate}
\item Let $t>0$.  First, we prove that  $\rho_{\epsilon,t}^{(n)}(k) =\frac{1}{nk}| \{x\in \KN,\; |C^{(n)}_{nt}(x)| = k\}|$ converges in $L^2$ to  
$\rho_{\epsilon,t}(k) =\frac{1}{k}\Pd(T^{(1)}_{\epsilon, t}= k)$.
Theorem \ref{thmsub} and Corollary \ref{coroltwocomponents} imply the
convergence of the first two moments of $\rho_{\epsilon,t}^{(n)}(k)$ to 
$\rho_{\epsilon,t}(k)$ and $(\rho_{\epsilon,t}(k))^2$ respectively and thus the
$L^2$ convergence of $(\rho_{\epsilon,t}^{(n)}(k))_n$.  Indeed, 
$\Ed(\rho_{\epsilon,t}^{(n)}(k))=\frac{1}{k}\Pd(|C^{(n)}_{nt,\epsilon}(1)|=k)$
converges to $\rho_{\epsilon,t}(k)$. The second moment is
$$\Ed((\rho_{\epsilon,t}^{(n)}(k))^2)=\frac{1}{nk^2}\Pd(|C^{(n)}_{nt,\epsilon}
(1)|=k)+(1-\frac{1}{n})\frac{1}{k^2}\Pd(|C^{(n)}_{nt,\epsilon}(1)|=k\; \et\;
|C^{(n)}_{nt,\epsilon}(2)|=k).$$ 
The first term converges to 0 and the second term converges to
$(\rho_{\epsilon,t}(k))^2$. 
\item 
It remains to show  that $\{\rho_{\epsilon,t},\; t\in\RR_+\}$ is solution of the coagulation equations:  
\[
\frac{d}{dt}\rho_t(k) =
\sum_{j=2}^{+\infty}\frac{1}{(\epsilon+1)^j}G_{j}(\rho_t,k)
\]
where 
\begin{multline*} G_{j}(\rho_t,k)=
\frac{1}{j}\Big(\sum_{\substack{(i_1,\ldots,i_{j})\in(\NN^*)^{j}\\  i_1+\cdots
+i_j=k}}\prod_{u=1}^{j}i_u\rho_{t}(i_u)\Big)\un_{\{j\leq
k\}}-k\rho_t(k)\Big(\sum_{i=1}^{+\infty}
i\rho_t(i)\Big)^{j-1}\\
-k\rho_{t}(k)\sum_{h=1}^{j-1}\binom{j-1}{h}
\Big(\sum_{i=1}^{+\infty}i(\rho_{0}(i)-\rho_t(i))\Big)^h 
\Big(\sum_{u=1}^{+\infty}u\rho_t(u)\Big)^{j-1-h}.
\end{multline*}
By definition of $\rho_{\epsilon,t}$, 
$G_{j}(\rho_{\epsilon,t},k)=\frac{1}{j}\Pd(T^{(j)}_{\epsilon,t}=k)-k\rho_{\epsilon,t}(k)$ where $T^{(j)}_{\epsilon,t}$ is the total progeny of a Galton-Watson process with family size distribution 
$\CPois(\frac{t}{\epsilon(\epsilon+1)},\mathcal{G}_{\NN^*}(\frac{\epsilon}{
\epsilon+1}))$  and $j$ ancestors. \\
The probability distribution of $T^{(j)}_{\epsilon,t}$ is computed in the appendix (Lemma \ref{lemTdistr}):
\[
\left\{\begin{array}{l}P(T^{(j)}_{\epsilon,t} = j) = \displaystyle{e^{-\frac{j
t}{\epsilon(\epsilon+1)}}}\\
P(T^{(j)}_{\epsilon,t} = k) = \displaystyle{\frac{j}{k}\frac{e^{-\frac{k
t}{\epsilon(\epsilon+1)}}}{(\epsilon+1)^{k-j}}\sum_{h=1}^{k-j}\binom{k-j-1}{h-1}
\frac{1}{h!}\Big(\frac{kt}{\epsilon+1}\Big)^h}\quad \forall k\geq j+1.  
\end{array}\right.
\]
We deduce that
$$\sum_{j=2}^{+\infty}\frac{1}{(\epsilon+1)^j}G_j(\rho_{\epsilon,t},k)=\frac{e^{-\frac{tk}{
\epsilon(\epsilon+1)}}}{k(\epsilon+1)^k}\left(1+\sum_{h=1}^{k-2}\frac{1}{h!}\Big(\frac{tk}{
\epsilon+1}\Big)^{h}\sum_{j=2}^{k-h}\binom{k-j-1}{h-1}\right)-\frac{k}{
\epsilon(\epsilon+1)}\rho_{\epsilon,t}(k).$$
By using that $\displaystyle{\binom{m}{k}=\sum_{j=k-1}^{m-1}\binom{j}{k-1}}$ for every $1\leq
k\leq m-1$, we obtain that
\linebreak[4] $\displaystyle{\sum_{j=2}^{+\infty}\frac{1}{(\epsilon+1)^j}G_j(\rho_{\epsilon,t},k)}$ is equal to 
$\frac{d}{dt}\rho_{\epsilon,t}(k).$\qed
\end{enumerate}
\appendix

\section{Some properties of the Galton-Watson process with offspring
distribution $\CPois(\lambda,\mathcal{G}_{\NN^*}(p))$\label{propGW}}
Let $\lambda$ be a positive number and let $p\in]0,1[$. In this section we
describe some properties of the Galton-Watson process with offspring
distribution $\CPois(\lambda,\mathcal{G}_{\NN^*}(p))$ which are useful in the
study of the component sizes of a random graph.  In our model, the parameters
are $\lambda=\frac{t}{\epsilon(\epsilon+1)}$ and
$p=\frac{\epsilon}{\epsilon+1}$.  
\paragraph{Average number of offspring.} 
 The expectation of the compound Poisson distribution\linebreak[4]
$\CPois(\lambda,\mathcal{G}_{\NN^*}(p))$ is $\frac{\lambda}{p}$, hence the
Galton-Watson process with offspring distribution
$\CPois(\lambda,\mathcal{G}_{\NN^*}(p))$ is subcritical if
$\frac{\lambda}{p}<1$. 
\paragraph{Extinction probability.}
Let $\rho$ denote the extinction probability of this Galton-Watson process with
one ancestor. As the   probability-generating function of
$\CPois(\lambda,\mathcal{G}_{\NN^*}(p))$ is
$$\phi(s)=\exp(-\lambda\frac{1-s}{1-s+sp})\quad \forall s<\frac{1}{1-p},$$ $\rho$ is
the smallest positive solution to the equation 
$\exp(-\lambda\frac{1-s}{1-s+sp})=s$. 
\paragraph{Total progeny distribution.}
Let us first recall the general result on the total population size of a
Galton-Watson  proved by Dwass in \cite{Dwass}. 
\begin{bibli}
\label{thDwass}
 Consider a branching process with offspring distribution $\nu$ and $u\geq 1$
ancestors. Let 
$T$ denote its total progeny and let $(X_n)_n$ be a sequence of independent
random variables with distribution $\nu$. 
 $$\forall k\geq u,\ \Pd(T=k)=\frac{u}{k}P(X_1+\ldots+X_k=k-u). $$
\end{bibli}
Recall that in the supercritical case (i.e. $\sum_{k}\nu(k)>1$),  $P(T<+\infty)=\rho^u<1$ if $\rho$ denote the extinction probability of the branching process starting from one ancestor. \\ 
Using this theorem, we obtain:
\begin{lem}\label{lemTdistr} 
Let $T^{(u)}$ denote the total progeny  of a Galton-Watson process with $u$
ancestors and with offspring distribution 
$\CPois(\lambda,\mathcal{G}_{\NN^*}(p))$. 
Then,  
\begin{equation}
\label{TGWdistrib}
\left\{\begin{array}{l}P(T^{(u)}= u) = \displaystyle{e^{-u \lambda}}\\
P(T^{(u)} = k) = \displaystyle{\frac{u}{k}e^{-k
\lambda}(1-p)^{k-u}\sum_{j=1}^{k-u}\binom{k-u-1}{j-1}\frac{1}{j!}\Big(\frac{
k\lambda p}{1-p}\Big)^j}\quad \forall k\geq u+1.  \\
\end{array}\right.
\end{equation}
\end{lem}
\begin{proof}
In our setting, the  sum $X_1+\ldots+X_k$ appearing in the Dwass's theorem 
  has the same  distribution as $\sum_{i=0}^{Z_k}Y_i$  where $Z_k$ is a
$\Pois(k\lambda)$-distribution random variable and $(Y_i)_i$ is a
sequence of independent random variables with 
$\mathcal{G}_{\NN^*}(p)$-distribution. \\
Therefore,
$\Pd(T^{(u)}=u)=\Pd(Z=0)=\displaystyle{e^{-u \lambda}}$ and for every $k\geq
u+1$, $$\Pd(T^{(u)}=k)=\frac{u}{k}\sum_{j=1}^{k-u}\Pd(Z=j)\Pd(Y_1+\ldots +
Y_j=k-u)$$ with 
$\Pd(Y_1+\ldots + Y_j=k-u)=\binom{k-u-1}{j-1}(1-p)^{k-u-j}p^j$. 
\end{proof}
\paragraph{Dual Galton-Watson process. }
A supercritical Galton-Watson process conditioned to become extinct 
is a subcritical Galton-Watson process: 
\begin{bibli}[\cite{AthreyaNey72}, Theorem 3, p. 52]
 Let $(Z_n)_n$ be a supercritical Galton-Watson process with one ancestor. 
 Let $\phi$ denote the generating function of its offspring distribution and let 
 $q$ denote its extinction probability. 
 Assume that $\phi(0)>0$. Then,  $(Z_n)_n$ conditioned to become extinct has the same law
as a subcritical Galton-Watson process with one ancestor and offspring generating function
$s\mapsto\frac{1}{q}\phi(qs)$. 
\end{bibli}
If the offspring distribution is $\CPois(\lambda,\mathcal{G}_{\NN^*}(p))$, we obtain:
\begin{lem}
 \label{dualproces}Let $Z$ be a Galton-Watson process with offspring
distribution 
$\CPois(\lambda,\mathcal{G}_{\NN^*}(p))$. 
Assume that $\frac{\lambda}{p}> 1$ and let $q$ denote the extinction
probability  of $Z$.  Then $Z$ conditioned to become extinct has the same law
as the subcritical Galton-Watson process with family size distribution 
$\CPois(\tilde{\lambda},\mathcal{G}_{\NN^*}(\tilde{p}))$ where
$1-\tilde{p}=q(1-p)$ and  $\tilde{\lambda}=\lambda q \dfrac{p}{\tilde{p}}$.
\end{lem}
\noindent In particular, if $Z$ is the Galton-Watson process with family size
distribution $\CPois(\frac{t}{\epsilon(\epsilon+1)},\mathcal{G}_{\NN^*}(\frac{\epsilon}{
\epsilon+1}))$ and one ancestor  for $t > \epsilon^2$  then $Z$ conditioned to become extinct 
is a subcritical Galton-Watson
process with family size distribution 
 $\CPois(\frac{t}{q_{\epsilon,t}\tilde{\epsilon}(\tilde{\epsilon}+1)},\mathcal{G}_{\NN^*}
(\frac{\tilde{\epsilon}}{\tilde{\epsilon}+1}))$ where  
$\tilde{\epsilon}+1=\frac{\epsilon+1}{q}$. 
\\
Let us note that in the Erd\"os-R\'enyi random graph $\ER(n,\frac{t}{n})$ with $t>1$,  
the `dual' of the Galton-Watson process with Poisson$(t)$ offspring distribution 
corresponds to the limit of the Galton-Watson process 
associated to the component of a vertex outside the `maximal component' of $\ER(n,\frac{t}{n})$
(see, for example, \cite{Spencer91}). 
It is also the case for the random graph $\mathcal{G}^{(n)}_{nt}$  
defined by the loop set
$\DL^{(n)}_{nt}$ for $t >\epsilon^2$. Indeed, if $A\subset \KN$ is a subset of size 
$m_n\sim n q_{\epsilon,t}$, then the loop set inside $A$ at time 
$nt$ (denoted by $\DL^{(n)}_{nt}(A)$) has the same law as a Poisson loop set defined on 
$\bar{K}_{m_n}$ endowed with unit conductances and a uniform killing measure with intensity 
$\tilde{\kappa}_n=m_n\epsilon_n$ 
at time $m_nt_n$ where  
$\epsilon_n=\epsilon+(\frac{n}{m_n}-1)(\epsilon+1)\underset{n\rightarrow +\infty}{\rightarrow}\tilde{\epsilon}$ and 
$t_n=\frac{nt}{m_n}\underset{n\rightarrow +\infty}{\rightarrow}\frac{t}{q_{\epsilon,t}}$.  
%
%%%%%%%%%%%%%%%%%%%%%%%%%%%%%%%%%%%%%%%%%%%%%%%%%%%%
\section{The random graph process defined by loops of fixed length\label{section:fixedlength}}
Let $j$ be an integer greater than or equal to 2. In this section, we consider the random graph  $\mathcal{G}_{t}^{(n,j)}$ defined by the set of loops of length $j$ at time $t$. The study of $\mathcal{G}_{t}^{(n)}$ detailed in the paper can be conducted in the same manner on $\mathcal{G}_{t}^{(n,j)}$. We present in this section the main results and justify Proposition \ref{prop:hydrodynj}. \\
For a vertex $x$, let $C_{t}^{(n,j)}(x)$ denote the connected component of a vertex $x$ in $\mathcal{G}_{t}^{(n,j)}$ ans set $\NL_{t,x}^{(j)}(\KN)=\{\ell \in\NL_{t,x}(\KN),\ |\ell|=j\}$. A Galton-Watson process can be constructed by using the component exploration procedure described in Section \ref{sect:exploration} to  explore $C_{t}^{(n,j)}(x)$: the offspring distribution of this Galton-Watson process is the distribution of $(j-1)|\NL_{t,x}^{(j)}(\KN)|$.  Set $\beta_{n,\epsilon}^{(j)}=\mu(\NL_{t,x}^{(j)}(\KN))$: $$\beta_{n,\epsilon}^{(j)}=\frac{1}{j(\epsilon+1)^j}(1-(1-\frac{1}{n})^j). $$
The random variable  $|\NL_{t,x}^{(j)}(\KN)|$ is a $\Pois(t\beta_{n,\epsilon}^{(j)})$-distributed random variable. Therefore, $|\NL_{tn(\epsilon+1)^j,x}^{(j)}(\KN)|$ converges in distribution to $\Pois(t)$ as $n$ tends to $+\infty$. 
Let $\nu_{t,j}$ be the distribution of $(j-1)Y$ where $Y$ denotes a \Pois$(t)$-distributed random variable. Let $T^{(u,j)}_t$ denote the total progeny of a Galton-Watson process with $u$ ancestors and  offspring distribution $\nu_{t,j}$. 
The size of the connected component of $x$ in $\mathcal{G}_{tn(\epsilon+1)^j}^{(n,j)}$ can be compared to $T^{(1,j)}_t$:
\begin{thm}
\label{sizecompj}
Let $\epsilon$ and  $t$ be two positive reals. \\
If  $(k_n)_n$ is a  sequence of positive numbers such that $\dfrac{k^{2}_n}{n}$
converges to $0$, then   
 $$\Pd(|C^{(n,j)}_{nt(\epsilon+1)^j}(x)|\leq k_n)-\Pd(T^{(1,j)}_{t}\leq k_n)$$ 
converges to 0. 
\end{thm}
We deduce the following joint limit theorem:
\begin{corol}
\label{coroltwocompj}
Let $x$ and $y$ be two distinct vertices of $\bar{K}_n$. 
For every $t>0$, $k, h\in\NN^*$, $\Pd(|C^{(n,j)}_{nt(\epsilon+1)^j}(x)|=k \et |C^{(n,j)}_{nt(\epsilon+1)^j}(y)|=h)$ converges to $\Pd(T^{(1,j)}_{t}=k)\Pd( T^{(1,j)}_{t}=h)$ as $n$ tends to $+\infty$. 
\end{corol}
\paragraph{Proof of Proposition \ref{prop:hydrodynj}.}
\begin{enumerate} \item Let $k$ be a positive integer. The average number of components of size $k$ in the random graph  $\mathcal{G}_{tn(\epsilon+1)^j}^{(n,j)}$ is $\displaystyle{\rho^{(n,j)}_{\epsilon,t}(k)=\frac{1}{nk}|\{x\in \KN,\ |C^{(n,j)}_{nt(\epsilon+1)^j}(x)| = k\}|}$.\\ Set $\rho^{(j)}_{t}(k)=\frac{1}{k}\Pd(T^{(1,j)}_t=k)$. By Theorem \ref{sizecompj} and Corollary \ref{coroltwocompj}, the first two moments of $\rho^{(n,j)}_{\epsilon,t}(k)$ converge to $\rho^{(j)}_{t}(k)$ and $(\rho^{(j)}_{t}(k))^2$ respectively.  Therefore, $(\rho^{(n,j)}_{\epsilon,t}(k))_n$ converges to $\rho^{(j)}_{t}(k)$ in $L^2$ as $n$ tends to $+\infty$. 
 \item 
To complete the proof of Proposition \ref{prop:hydrodynj}, we compute the distribution of $T^{(u,j)}_{t}$ for $u\in\NN^*$, using Dwass's Theorem: 
$$\left\{\begin{array}{ll}
 \Pd(T^{(u,j)}_t=u+(j-1)k)=\frac{u}{k!}(u+(j-1)k)^{k-1}t^{k}e^{-(u+(j-1)k)t}&\quad \forall k\in\NN \\
 \Pd(T^{(u,j)}_t=h)=0 &\text{ if } h-u\not\in (j-1)\NN.
 \end{array}\right.
$$
As $1+(j-1)k=j+(j-1)(k-1)$, $$\frac{d}{dt}\rho^{(j)}_t(1+(j-1)k)=\frac{1}{j}P(T^{(j,j)}_t=1+(j-1)k)-\Pd(T^{(1,j)}_t=1+(j-1)k)$$
Therefore, $(\rho^{(j)}_t(k))_{t\geq 0}$ is solution of equation \eqref{coageqj} for every $k\in\NN^*$.
\end{enumerate}
%

%
%\bibliographystyle{plain}
%\bibliography{../loopgraph}
%
\end{document}